\documentclass[12pt,a4paper]{article}
\usepackage{amssymb,amsmath,amsthm}
\usepackage{latexsym}

\newtheorem{theorem}{Theorem}
\newtheorem{lemma}{Lemma}
\newtheorem{proposition}{Proposition}

\newtheorem{remark}{Remark}
\numberwithin{equation}{section}
\numberwithin{theorem}{section}
\numberwithin{lemma}{section}
\numberwithin{proposition}{section}
\numberwithin{corollary}{section}
\numberwithin{remark}{section}

\begin{document}
\title{Asymptotic structure of steady Stokes flow \\
around a rotating obstacle in two dimensions}
\author{Toshiaki Hishida\thanks
{Dedicated to Professor Reinhard Farwig on his 60th birthday}}
\date{Graduate School of Mathematics, Nagoya University \\
Nagoya 464-8602 Japan \\
{\tt hishida@math.nagoya-u.ac.jp}}
\maketitle
\begin{abstract}
This paper provides asymptotic structure at spatial infinity of 
plane steady Stokes flow in exterior domains when the obstacle is rotating
with constant angular velocity.
The result shows that there is no longer Stokes' paradox due to the
rotating effect.
\end{abstract}

\section{Introduction}\label{intro}

Let $\Omega$ be an exterior domain in the plane $\mathbb R^2$ with
smooth boundary $\partial\Omega$, and consider the motion of
a viscous incompressible fluid around an obstacle (rigid body)
$\mathbb R^2\setminus\Omega$.
As compared with 3D problem,
we have less knowledge about exterior steady flows in 2D
despite efforts of several authors mentioned below.
The difficulty is to analyze the asymptotic behavior of the flow 
at infinity.
This is related to the following hydrodynamical paradox found 
by Stokes (1851):
The problem
\begin{equation}
-\Delta u+\nabla p=0, \qquad\mbox{div $u$}=0 \qquad\mbox{in $\Omega$},
\label{stokes}
\end{equation}
\begin{equation}
u|_{\partial\Omega}=0, \qquad u\to u_\infty\quad\mbox{as $|x|\to\infty$}
\label{bc}
\end{equation}
admits no solution unless $u_\infty=0$, where
$u(x)=(u_1,u_2)^T$ and $p(x)$ denote the velocity and pressure,
respectively, of the fluid.
Throughout this paper, all vectors are column ones and
$(\cdot)^T$ denotes the transpose of vectors or matrices.
Later on, Chang and Finn \cite{CF} made it clear that the Stokes paradox
is interpreted in terms of the total net force exerted by the fluid
to the obstacle
\begin{equation}
N:=\int_{\partial\Omega}T(u,p)\nu\,d\sigma,
\label{net}
\end{equation}
where $T(u,p)$ is the Cauchy stress tensor given by
\begin{equation}
T(u,p):=\big(T_{jk}(u,p)\big)=Du-p\mathbb I, \quad
Du:=\nabla u+(\nabla u)^T, \quad \mathbb I=\left(\delta_{jk}\right),
\label{stress}
\end{equation}
and $\nu$ denotes the outward unit normal to $\partial\Omega$;
in fact,
they proved that the flow satisfying \eqref{stokes} can be bounded
at infinity only if the net force \eqref{net} vanishes.
This is an immediate consequence of asymptotic representation
at infinity of solutions to \eqref{stokes} due to themselves \cite{CF},
see \eqref{stokes-expan}.
The original form of the Stokes paradox mentioned above follows from
the result of Chang and Finn as a corollary
because the net force \eqref{net} never vanishes provided that
$\{u,p\}$ is nontrivial and satisfies \eqref{stokes} together with
$u|_{\partial\Omega}=0$.
There are some other forms of the Stokes paradox, see
Galdi \cite[V.7]{Ga-b}, Kozono and Sohr \cite[Theorem A]{KS}.

For the case in which a constant velocity 
$u_\infty\in\mathbb R^2\setminus\{0\}$ is prescribed at infinity
or, equivalently, the obstacle is translating with velocity $-u_\infty$
(while the flow is at rest at infinity),
Oseen (1910) proposed his linearization of 
the Navier-Stokes system around $u_\infty$ to get around the Stokes paradox.
This works well because the fundamental solution
of the Oseen operator
$-\Delta u+u_\infty\cdot\nabla u+\nabla p$
possesses some decay structure with wake,
while the Stokes fundamental solution grows logarithmically at infinity,
see \eqref{stokes-funda}.
Finn and Smith \cite{FS1}, \cite{FS2}, \cite{Sm} actually adopted 
the Oseen linearization to succeed in construction of the Navier-Stokes flow
when $u_\infty$ is not zero but small enough (and the external force is small,
too, unless it is absent).
Later on, Galdi \cite{Ga93} refined the result
by means of $L^q$-estimates, see also \cite[Section XII.5]{Ga-b}.
The similar existence theorem for the case 
$u_\infty=0$ is still an open question
even for small external force.
Even before the results mentioned above,
Leray \cite{L} constructed at least one Navier-Stokes flow
with finite Dirichlet integral
without any smallness condition, however,
the asymptotic behavior at infinity of his solution
is still unclear
and all the related results obtained so far are partial answers
(Gilbarg and Weinberger \cite{GW1}, \cite{GW2} and Amick \cite{A}, \cite{A2}).
For details, see Galdi \cite{Ga}, \cite{Ga-3} and \cite{Ga-b}.
It should be noted that symmetry helps to attain
the boundary condition
$u\to 0$ at infinity,
see \cite{Ga-3}, \cite{Y}, \cite{PiRu} and the references therein.  
Among them, Yamazaki \cite{Y} employed a linearization method to
construct a small Navier-Stokes flow
decaying like $|x|^{-1}$ at infinity under a sort of symmetry;
indeed, the symmetry he adopted enables us to avoid the Stokes paradox
since the net force vanishes.

In this paper it is shown that, instead of the translation
mentioned above,
the rotation of the obstacle leads to the resolution of the
Stokes paradox in the sense that:
(i) The flow can be bounded 
(and even goes to a constant vector at the rate $|x|^{-1}$)
at infinity even if the net force
\eqref{net} does not vanish (Theorem \ref{decay});
(ii) Given external force decaying sufficiently fast,
there exists a linear flow
which enjoys $u(x)=O(|x|^{-1})$ as $|x|\to\infty$ (Theorem \ref{existence}).
We also provide a remarkable asymptotic representation of the flow
at infinity,
in which the leading term involves the rotational profile
$x^\perp/|x|^2$ whose coefficient is given by the torque,
where $x^\perp=(-x_2,x_1)^T$,
see \eqref{asym-rep} and \eqref{rot-u-expan}.
Here, the linear system arising from the flow around a rotating
obstacle with constant angular velocity 
$a\in\mathbb R\setminus\{0\}$
is described as
\begin{equation}
-\Delta u-a\left(x^\perp\cdot\nabla u-u^\perp\right)+\nabla p=f, \qquad
\mbox{div $u$}=0 \qquad\mbox{in $\Omega$}
\label{rot-stokes}
\end{equation}
in the reference frame attached to the obstacle.
We recall the derivation of \eqref{rot-stokes} in the next section.

The essential reason why there is no longer Stokes' paradox is 
asymptotic structure of the fundamental solution of the system
\eqref{rot-stokes} in the whole plane $\mathbb R^2$.
Roughly speaking, the oscillation caused by rotation yields better
decay structure of the fundamental solution, from which
combined with some cut-off techniques we obtain the main results.
It is worth while comparing with the result \cite{FH1} 
by Farwig and the present author on the 3D Stokes flow around
a rotating obstacle, in which the axis of rotation
($e_3$-axis without loss, where $e_3=(0,0,1)^T$)
plays an important role;
in fact, $e_3\cdot N$ controls the rate of decay.
The result would suggest better decay studied in this paper
since we do not have the axis of rotation in 2D, however,
there are some difficulties compared with 3D case.
Look at \eqref{fundamental} below,
which would be heuristically a fundamental solution,
but this is by no means trivial because of lack of absolute convergence
unlike 3D case.
We thus employ the centering technique due to Guenther and Thomann \cite{GT},
that is, we subtract the worst part, 
whose time-integral converges on account of
oscillation, from the integrand of \eqref{fundamental} such that the remaining part
converges absolutely and can be treated in a usual way.
This technique is also needed to justify some estimates of 
the fundamental solution, see Lemma \ref{loc-sum}.
Asymptotic analysis of the fundamental solution
to find the asymptotic representation \eqref{asym-funda}
is similar to the argumant adopted for 3D (\cite{FH1}),
in which, however, the external force
$f$ is assumed to have a compact support.
In this paper we will derive further properties of the fundamental solution
and the corresponding volume potential \eqref{vol}
to deal with the external force decaying sufficiently fast for
$|x|\to\infty$, see \eqref{f-cond} and \eqref{f-cond-2}.

This paper is a step toward analysis of
the Navier-Stokes flow around a rotating obstacle in the plane.
To proceed to the nonlinear case, 
it is reasonable to consider the external force
$f=\mbox{div $F$}$ with $F(x)=O(|x|^{-2})$
in view of the nonlinear structure
$u\cdot\nabla u=\mbox{div $(u\otimes u)$}$, see Remark \ref{rem-grad}.
This will be discussed in a forthcoming paper.
As for asymptotic structure of the Navier-Stokes flow
around a rotating obstacle in 3D, the leading term
at infinity was found first by \cite{FH2} and then 
the estimate of the remainder was refined by \cite{FGK}.

When the rotating obstacle is the two-dimensional disk 
and the external force is absent,
the Navier-Stokes system subject to the no-slip boundary condition
\eqref{noslip} admits an explicit solution
\eqref{yamazaki-ns} in the original frame, see Galdi \cite[p.302]{Ga-b}.
Recently, Hillairet and Wittwer \cite{HW} considered small perturbation
from this solution with large angular velocity $|a|$ to find the Navier-Stokes
flow decaying like $|y|^{-1}$ at infinity,
whose leading profile is given by $y^\perp/|y|^2$.
See also Guillod and Wittwer \cite[Section 4]{GW},
who provided among others numerical simulations of the related issue.

This paper is organized as follows.
In the next section, after recalling the equations in the reference frame,
we present the main theorems.
Section 3 is essentially the central part of this paper and
we carry out a detailed analysis of several asymptotic properties of the 
fundamental solution of \eqref{rot-stokes} in the whole plane $\mathbb R^2$.
Section 4 is devoted to decay structure of the system \eqref{rot-stokes}
to prove Theorem \ref{decay}.
In the final section we show the existence of a unique linear flow 
which goes to zero as $|x|\to\infty$ to prove Theorem \ref{existence}.

\section{Results}\label{result}

We begin with introducing notation.
Set
$B_\rho(x_0)=\{x\in\mathbb R^2; |x-x_0|<\rho\}$, where
$x_0\in\mathbb R^2$ and $\rho >0$.
Given exterior domain $\Omega$ with smooth boundary $\partial\Omega$,
we fix $R\geq 1$ such that
$\mathbb R^2\setminus \Omega\subset B_R(0)$.
For $\rho\geq R$ we set 
$\Omega_\rho=\Omega\cap B_\rho(0)$.
Let $D$ be one of $\Omega, \mathbb R^2$ and $\Omega_\rho$, 
and let $1\leq q\leq \infty$.
We denote by $L^q(D)$ the usual Lebesgue space with norm 
$\|\cdot\|_{L^q(D)}$.
It is also convenient to introduce the weak-$L^2$ space
$L^{2,\infty}(D)$ (one of the Lorentz spaces, see \cite{BL}) by
$L^{2,\infty}(D)=\left(L^1(D), L^\infty(D)\right)_{1/2,\infty}$
with norm $\|\cdot\|_{L^{2,\infty}(D)}$,
where $(\cdot,\cdot)_{1/2,\infty}$ stands for the real interpolation functor.
The measurable function $f$ belongs to $L^{2,\infty}(D)$ if and only if
$\sup_{\tau>0}\tau|\{x\in D; |f(x)|>\tau\}|^{1/2}<\infty$,
where $|\cdot|$ denotes the Lebesgue measure.
Note that $L^2(D)\subset L^{2,\infty}(D)$; indeed,
$|x|^{-1}\in L^{2,\infty}(D)$.
By $H^k(D)$ and $H_0^1(D)$ we respectively denote the 
$L^2$-Sobolev space of $k$-th order ($k\geq 1$) with norm
$\|\cdot\|_{H^k(D)}$
and the completion of $C_0^\infty(D)$
(consisting of smooth functions with compact support)
in $H^1(D)$.
We use the same symbol for denoting the spaces
of scalar, vector and tensor valued functions.

Before stating our results, we briefly explain the derivation of the
system \eqref{rot-stokes} for the readers' convenience
although that is the same as in 3D case (\cite{Ga-2}, \cite{H}).
Suppose a compact obstacle (rigid body)
$\mathbb R^2\setminus\Omega$ is rotating about the origin in the
plane with constant angular velocity
$a\in\mathbb R\setminus\{0\}$,
and let us start with the nonstationary Navier-Stokes system
\[
\partial_tv+v\cdot\nabla v=\Delta v-\nabla q+g, \qquad
\mbox{div $v$}=0
\]
in the time-dependent exterior domain
$\Omega(t)=\{y=O(at)x;\, x\in\Omega\}$ with
\[
O(t)=\left(
\begin{array}{cc}
\cos t & -\sin t \\
\sin t & \cos t
\end{array}
\right),
\]
where $v(y,t)$ and $q(y,t)$ are unknowns, while $g(y,t)$ is a given
external force.
The fluid velocity is assumed to attain the rotational velocity
of the rigid body
on the boundary $\partial\Omega(t)$ (no-slip condition),
while it is at rest at infinity, that is,
\[
v|_{\partial \Omega(t)}=ay^\perp, \qquad
v\to 0 \quad\mbox{as $|y|\to\infty$}.
\]
We take the frame attached to the obstacle by making change of variables
\begin{equation}
\begin{split}
&y=O(at)x, \qquad
u(x,t)=O(at)^Tv(y,t), \qquad
p(x,t)=q(y,t), \\
&f(x,t)=O(at)^Tg(y,t),
\end{split}
\label{transform}
\end{equation}
so that the equation of momentum is reduced to 
\begin{equation*}
\begin{split}
\partial_tu
&=O(at)^T\partial_tv
+O(at)^T\left(a\,\dot{O}(at)x\right)\cdot\nabla_yv
+a\,\dot{O}(at)^Tv \\
&=O(at)^T(-v\cdot\nabla_y v+\Delta_yv-\nabla_yq+g)
+a\left(x^\perp\cdot\nabla_xu-u^\perp\right) \\
&=-u\cdot\nabla_x u+\Delta_xu-\nabla_xp+f
+a\left(x^\perp\cdot\nabla_xu-u^\perp\right)
\end{split}
\end{equation*}
in $\Omega$,
where $\dot{O}(t)=\frac{d}{dt}O(t)$.
If $f$ is independent of $t$, then one can consider the steady problem
\begin{equation}
-\Delta u-a\left(x^\perp\cdot\nabla u-u^\perp\right)+\nabla p+u\cdot\nabla u=f,
\qquad\mbox{div $u$}=0 \qquad\mbox{in $\Omega$}
\label{rot-ns}
\end{equation}
subject to
\begin{equation}
u|_{\partial\Omega}=ax^\perp, \qquad u\to 0\quad\mbox{as $|x|\to\infty$}.
\label{rot-bc}
\end{equation}
It is sometimes convenient to write the LHS of $\mbox{\eqref{rot-ns}}_1$
as divergence form
\begin{equation*}
\begin{split}
&\quad \Delta u+a\left(x^\perp\cdot\nabla u-u^\perp\right)-\nabla p
-u\cdot\nabla u  \\
&=\mbox{div $\big(S(u,p)-u\otimes u\big)$}
=\left(\sum_{k=1}^2\partial_k
\left\{S_{jk}(u,p)-u_ju_k\right\}\right)_{j=1,2}
\end{split}
\end{equation*}
with
\begin{equation}
S(u,p)=\big(S_{jk}(u,p)\big)=
T(u,p)+a\left(u\otimes x^\perp-x^\perp\otimes u\right)
\label{new-stress}
\end{equation}
where $T(u,p)$ is given by \eqref{stress},
$u\otimes v=(u_jv_k)$ stands for the matrix for given 
vector fields $u$ and $v$, and $\partial_k=\partial_{x_k}$.

The only problem we intend to address in this paper is
the associated linear system \eqref{rot-stokes}.
On account of the relation 
\begin{equation}
\int_\Omega\left[(x^\perp\cdot\nabla u-u^\perp)\cdot v +u\cdot(x^\perp\cdot\nabla v-v^\perp)\right]\,dx 
=\int_{\partial\Omega}(\nu\cdot x^\perp)(u\cdot v)\,d\sigma  
\label{skew}
\end{equation}
for vector fields $u$ and $v$
(so long as the calculation \eqref{skew-2} in section \ref{proof-1}
makes sense),
the opeartor $u\mapsto x^\perp\cdot\nabla u-u^\perp$
is skew-symmetric under the homogeneous boundary condition.
Also, by using the auxiliary function \eqref{auxi} below,
our problem with boundary condition $\mbox{\eqref{rot-bc}}_1$
can be reduced to the problem with the homogeneous one.
Hence, it is not hard to find at least one solution with
$\nabla u\in L^2(\Omega)$
for \eqref{rot-stokes} 
(even for the Navier-Stokes system \eqref{rot-ns} without restriction
on the size of $|a|$)
subject to the boundary condition
$u|_{\partial\Omega}=ax^\perp$ 
(only $\mbox{\eqref{rot-bc}}_1$)
along the same procedure as in Leray \cite{L}
provided $f=\mbox{div $F$}$ with
$F\in L^2(\Omega)$,
however, we do not know whether the behavior
$\mbox{\eqref{rot-bc}}_2$ at infinity is verified.
The asymptotic structure of this solution 
for $f$ decaying sufficiently fast at infinity
and, more generally,
that of $\{u,p\}$ satisfying \eqref{rot-stokes} without assuming any
boundary condition on $\partial\Omega$ are given by the following theorem.
For simplicity we are concerned with smooth solutions
although the result can be extended to less regular solutions
(in view of Proposition \ref{justify} for the whole plane problem).
\begin{theorem}
Let $a\in\mathbb R\setminus\{0\}$.
Suppose that
$\{u,p\}\in H^1_{loc}(\overline{\Omega})\times L^2_{loc}(\overline{\Omega})$
is a smooth solution to the system \eqref{rot-stokes}
with $f\in C^\infty(\Omega)$ satisfying
\begin{equation}
\int_\Omega |x||f(x)|\,dx<\infty, \qquad
|f(x)|\leq\frac{C}{(1+|x|^3)\big(\log\, (e+|x|)\big)}
\label{f-cond}
\end{equation}
where the constant $C>0$ is independent of $x\in\Omega$.
Assume either

\emph{(i)}
$\nabla u\in L^r(\Omega\setminus B_R(0))$ 
for some $r\in (1,\infty)$

\noindent
or

\emph{(ii)}
$u(x)=o(|x|)$ as $|x|\to\infty$.

\noindent
Then there are constants 
$u_\infty\in\mathbb R^2$ and $p_\infty\in\mathbb R$ such that:
\begin{enumerate}
\item
{\em (asymptotic behavior)}
\begin{equation}
\left\{
\begin{aligned}
&u(x)=u_\infty+ (1+|a|^{-1})\,O(|x|^{-1}), \\
&p(x)=-a\,u_\infty^\perp\cdot x+p_\infty+O(|x|^{-1}),
\end{aligned}
\right.
\label{resolution}
\end{equation}
as $|x|\to\infty$.
\item
{\em (energy balance)}

We have 
$\nabla u\in L^2(\Omega)$
(even if we do not assume \emph{(i)} with $r=2$) and
\begin{equation}
\begin{split}
\frac{1}{2}\int_\Omega |Du|^2dx
&=\int_{\partial\Omega}\left[\big(\widetilde T(u,p)\nu\big)\cdot(u-u_\infty)
+\frac{a\,(\nu\cdot x^\perp)}{2}|u-u_\infty|^2\right]d\sigma  \\
&\qquad +\int_\Omega f\cdot(u-u_\infty)\,dx
\end{split}
\label{energy}
\end{equation}
with
$\widetilde T(u,p):=T(u,\;p+a\,u_\infty^\perp\cdot x-p_\infty)$,
where $Du$ and $T(\cdot,\cdot)$ are as in \eqref{stress}.
\item
{\em (asymptotic representation)}

If in addition
\begin{equation}
f(x)=o(|x|^{-3}(\log |x|)^{-1}) \quad
\mbox{as $|x|\to\infty$},
\label{f-cond-2}
\end{equation}
then
\begin{equation}
u(x)-u_\infty
=\frac{\alpha x^\perp-2\beta x}{4\pi |x|^2}+ (1+|a|^{-1})\,o(|x|^{-1}) \quad
\mbox{as $|x|\to\infty$},
\label{asym-rep}
\end{equation}
where
\begin{equation}
\begin{split}
&\alpha=\int_{\partial\Omega}
y^\perp\cdot\left\{\big(T(u,p)+a\,u\otimes y^\perp\big)\,\nu\right\}\,d\sigma_y
+\int_\Omega y^\perp\cdot f\,dy,  \\
&\beta=\int_{\partial\Omega}\nu\cdot u\,d\sigma.
\end{split}
\label{coe-asy}
\end{equation}
If in particular $f\in C_0^\infty(\overline{\Omega})$,
that is, the support of $f$ is compact in $\mathbb R^2$,
then the remainder decays like $O(|x|^{-2})$ in \eqref{asym-rep}.
\end{enumerate}
\label{decay}
\end{theorem}

Note that $T(u,p)\nu$ belongs to
$H^{-1/2}(\partial\Omega):=H^{1/2}(\partial\Omega)^*$
by the normal trace theorem since
$T(u,p)\in L^2(\Omega_R)$ and
$\mbox{div $T(u,p)$}=-a\big(x^\perp\cdot\nabla u-u^\perp\big)-f\in L^2(\Omega_R)$.
Therefore, the boundary integral
$\int_{\partial\Omega}y^\perp\cdot\big(T(u,p)\nu\big)\,d\sigma_y$
can be understood as
${}_{H^{1/2}(\partial\Omega)}\langle y^\perp, T(u,p)\nu\rangle_{H^{-1/2}(\partial\Omega)}$ in \eqref{coe-asy}.
Since $u\in H^{1/2}(\partial\Omega)$, the same reasoning as above justifies
$\int_{\partial\Omega}\big(\widetilde T(u,p)\nu\big)\cdot (u-u_\infty)\,d\sigma$ 
in \eqref{energy}.
All the other integrals in \eqref{energy} and \eqref{coe-asy} also make sense.

For the usual Stokes system \eqref{stokes},
under the same growth condition on $u(x)$ as in Theorem \ref{decay},
there is a constant $u_\infty\in\mathbb R^2$ such that
(Chang and Finn \cite[Theorem 1]{CF})
\begin{equation}
u(x)=u_\infty+E(x)N+O(|x|^{-1})
\label{stokes-expan}
\end{equation}
as $|x|\to\infty$, where
\begin{equation}
E(x)=\frac{1}{4\pi}
\left[\left(\log\frac{1}{|x|}\right)\mathbb I
+\frac{x\otimes x}{|x|^2}\right]
\label{stokes-funda}
\end{equation}
is the Stokes fundamental solution
and $N$ denotes the net force \eqref{net}.
We observe the remarkable difference between
$\mbox{\eqref{resolution}}_1$ and \eqref{stokes-expan};
in fact, the flow is bounded in Theorem \ref{decay}
even though the net force $N$ does not vanish.
We would say that this is the resolution of the Stokes paradox.

The leading term of \eqref{asym-rep} is of interest since it
contains the rotational profile $x^\perp/|x|^2$, which comes from
the leading term of the fundamental solution
of \eqref{rot-stokes}, see \eqref{asym-funda}.
The other profile 
$-x/(2\pi |x|^2)$ is called the flux carrier.
If in particular the flux $\beta$ at the boundary vanishes,
then the leading term is purely rotational and that is just the case
in the next theorem.
Look at the coefficient \eqref{coe-asy}
of $x^\perp/|x|^2$, where the integral
$\int_{\partial\Omega}y^\perp\cdot\big(T(u,p)\nu\big)\,d\sigma_y$
stands for the torque exerted by the fluid to the obstacle.
It is worth while noting that, in three dimensions,
one finds the rotational profile
$(e_3\times x)/|x|^3$,
whose coefficient involves the torque,
in the second term after the leading one.
For details, see Farwig and Hishida \cite[Theorem 1.1]{FH1}.
It is reasonable that both
$x^\perp/|x|^2=\nabla^\perp\log |x|$
and
$x/|x|^2=\nabla\log |x|$
are solutions to \eqref{rot-stokes} with $f=0$
in $\mathbb R^2\setminus\{0\}$ together with the constant pressure and, 
therefore, so is the leading term of \eqref{asym-rep}.
In fact, we observe
\[
\Delta\frac{x^\perp}{|x|^2}=0, \qquad
x^\perp\cdot\nabla\frac{x^\perp}{|x|^2}=\frac{(x^\perp)^\perp}{|x|^2}, \qquad
\mbox{div}\;\frac{x^\perp}{|x|^2}=0 \qquad
\mbox{in $\mathbb R^2\setminus\{0\}$}
\]
as well as \eqref{carr} (with $x_0=0$) below.

In Theorem \ref{decay} it is also possible to find the asymptotic 
representation of the pressure $p(x)$
without assuming any growth condition on $p(x)$ itself
since it can be controlled by the growth of $u(x)$
via the equation $\mbox{\eqref{rot-stokes}}_1$.
The leading profile of
$p(x)+a\,u_\infty^\perp\cdot x-p_\infty$
in $\mbox{\eqref{resolution}}_2$ is just the fundamental solution
$Q(x)=\frac{x}{2\pi |x|^2}$ of the pressure to the Stokes system.
This is because
\begin{equation}
\mbox{{div $\big(x^\perp\cdot\nabla u-u^\perp\big)$}}
=x^\perp\cdot\nabla\mbox{{div $u$}}=0
\label{solenoidal}
\end{equation}
so that the pressure part of the fundamental solution is independent
of $a\in\mathbb R$.
Thus we are not interested in the asymptotic representation of
the pressure, which the rotation of the obstacle does not affect so much.
The coefficient of the leading profile $Q(x)$ is rather complicated
in Theorem \ref{decay}, but it becomes
just the force in the next theorem, see \eqref{rot-p-expan}.

The next question is whether one can actually construct a solution
to \eqref{rot-stokes} when zero velocity is prescribed at infinity
as in \eqref{rot-bc}.
The following theorem gives an affirmative answer.

\begin{theorem}
Let $a\in\mathbb R\setminus\{0\}$.
Suppose that $f=\mbox{\emph{div} $F$}\in C^\infty(\Omega)$,
with $F\in L^2(\Omega)$, satisfies \eqref{f-cond}.
Then the system \eqref{rot-stokes} subject to \eqref{rot-bc} admits a 
smooth solution $\{u,p\}$, which is of class
$u\in L^{2,\infty}(\Omega)\cap H^1_{loc}(\overline{\Omega})$,
$p\in L^2_{loc}(\overline{\Omega})$
as well as
$\nabla u\in L^2(\Omega)$ and fulfills
\begin{equation}
\begin{split}
\|u\|_{L^{2,\infty}(\Omega)}
&\leq C\Big[
1+|a|+(1+|a|^{-1})\Big(\|F\|_{L^2(\Omega)} \\
&\qquad +\int_{\Omega}|x||f(x)|\,dx
+\sup_{x\in\Omega}|x|^3\big(\log\,(e+|x|)\big)|f(x)|\Big)\Big],  \\
&\quad \\
\|\nabla u\|_{L^2(\Omega)}
&\leq \|F\|_{L^2(\Omega)}+C|a|,
\end{split}
\label{weak-est}
\end{equation}
with some $C>0$ independent of $f$ and $a\in\mathbb R\setminus\{0\}$,
and $\{u,p\}$ enjoys all the assertions in Theorem \ref{decay} with
$\{u_\infty,p_\infty\}=\{0,0\}$.
In particular, we have
\begin{equation}
u(x)=\left(\int_{\partial\Omega}y^\perp\cdot\big(T(u,p)\nu\big)\,d\sigma_y
+\int_\Omega y^\perp\cdot f\,dy
\right)
\frac{x^\perp}{4\pi|x|^2}+(1+|a|^{-1})\,o(|x|^{-1}),
\label{rot-u-expan}
\end{equation}
\begin{equation}
p(x)=\left(\int_{\partial\Omega}T(u,p)\nu\,d\sigma
+\int_\Omega f\,dy
\right)\cdot
\frac{x}{2\pi|x|^2}+O(|x|^{-2}),
\label{rot-p-expan}
\end{equation}
as $|x|\to\infty$ under the additional condition \eqref{f-cond-2}
(which is needed only for \eqref{rot-u-expan}).
This is the only solution in the class
$\nabla u\in L^2(\Omega)$, $\{u,p\}\in L^2_{loc}(\overline{\Omega})$
with $\{u,p\}\to \{0,0\}$ as $|x|\to\infty$.
\label{existence}
\end{theorem}

Note that, when $a=0$, the problem is not always
solvable for given external force $f=\mbox{div $F$}$ even if
$F\in C_0^\infty(\overline\Omega)$, that may be also regarded
as the Stokes paradox.
The $L^\infty$-norm of $|x||u(x)|$ away from the boundary can be
also estimated by the RHS of $\mbox{\eqref{weak-est}}_1$
(see \eqref{far} for an approximate solution).
In order to control the $L^\infty$-norm of $u(x)$ near the boundary $\partial\Omega$,
the class $H^1_{loc}(\overline{\Omega})$ is not enough.
We put the term $x^\perp\cdot\nabla u-u^\perp$ in the RHS and use the
regularity theory of the usual Stokes system up to the boundary to show that
$u\in H^2_{loc}(\overline{\Omega})\subset L^\infty_{loc}(\overline{\Omega})$
together with a certain estimate,
which enables us to obtain the similar estimate of
$\sup_{x\in\Omega}(1+|x|)|u(x)|$ to $\mbox{\eqref{weak-est}}_1$.

We conclude this section with the following exact solutions
of both the Stokes and Navier-Stokes boundary value problems
without external force.
The Stokes flow \eqref{yamazaki-st} seems to be well known 
since it is found in some old literature.
The Navier-Stokes flow \eqref{yamazaki-ns}
is found in the second edition of \cite[p.302]{Ga-b}
(I learned it from Professor Masao Yamazaki around 2008).
Suppose the unit disk (rigid body) $\overline{B_1(0)}$
is rotating about the origin with constant angular velocity
$a\in\mathbb R\setminus\{0\}$.
Then the Navier-Stokes flow in the exterior 
$\Omega=\mathbb R^2\setminus\overline{B_1(0)}$ obeys
\begin{equation}
-\Delta v+\nabla q+v\cdot\nabla v=0, \qquad  
\mbox{div $v$}=0 \qquad\mbox{in $\Omega$}  
\label{disk-ns}
\end{equation}
subject to
\begin{equation}
v|_{\partial\Omega}=ay^\perp, \qquad 
v\to 0\quad \mbox{as $|y|\to\infty$} 
\label{noslip}
\end{equation}
and this problem has a solution
\begin{equation}
v(y)=\frac{ay^\perp}{|y|^2}, \qquad
q(y)=\frac{-a^2}{2|y|^2}+\mbox{constant}.
\label{yamazaki-ns}
\end{equation}
Also, the associated Stokes problem
\[
-\Delta v+\nabla q=0, \qquad  
\mbox{div $v$}=0 \qquad\mbox{in $\Omega$},
\]
sibject to \eqref{noslip} admits a solution
\begin{equation}
v(y)=\frac{ay^\perp}{|y|^2}, \qquad
q(y)=\mbox{constant}.
\label{yamazaki-st}
\end{equation}
Note that the Stokes flow \eqref{yamazaki-st} does not
contradict the Stokes paradox because 
$\int_{\partial\Omega} T(v,q)\nu\,d\sigma=0$
due to symmetry.
When the obstacle is a disk, 
we do not necessarily have to make the change of
variables \eqref{transform}, nevertheless, we can do so and
this case is not excluded in the present paper.
Steady flows in the original frame correspond to
time-periodic flows and are not steady in general
in the reference frame via \eqref{transform}.
But the Stokes flow
\eqref{yamazaki-st} becomes the steady one
$u(x)=ax^\perp/|x|^2,\, p(x)=$ constant in the reference frame as well
and this may be regarded as a special case
in Theorems \ref{decay} and \ref{existence} (when we take
$p=0$ in the latter theorem); indeed, one can verify
\[
\int_{\partial\Omega}y^\perp\cdot\big(T(u,p)\nu\big)\,d\sigma_y
=4\pi a
\]
in \eqref{rot-u-expan}.
Recently, Hillairet and Wittwer \cite{HW} proved that if the boundary value
$v|_{\partial\Omega}$ is sufficiently close to 
$ay^\perp$ with $|a|>\sqrt{48}$ in a sense and
$\int_{\partial\Omega}\nu\cdot v\,d\sigma=0$,
then the Navier-Stokes system \eqref{disk-ns} in the exterior
$\Omega=\mathbb R^2\setminus\overline{B_1(0)}$ of the disk
subject to this boundary condition admits at least one smooth solution,
which decays like $|y|^{-1}$ as $|y|\to\infty$.
The leading profile of their solution is given by
$y^\perp/|y|^2$ with some coefficient close to $a$.

\section{Fundamental solution}
\label{funda-sol}

In this section we derive the decay structure of the fundamental solution
of the linear system \eqref{rot-stokes} in the whole plane $\mathbb R^2$
when $a\in\mathbb R\setminus\{0\}$.
Because of \eqref{solenoidal} the pressure part 
of the fundamental solution is
\begin{equation}
Q(x-y)=\frac{x-y}{2\pi |x-y|^2},
\label{pre-funda}
\end{equation}
while the velocity part is given by
\begin{equation}
\Gamma_a(x,y)=\int_0^\infty O(at)^T\,K(O(at)x-y,t)\, dt,
\label{fundamental}
\end{equation}
where
\[
K(x,t)=G(x,t)\mathbb I+H(x,t)
\]
is the fundamental solution of unsteady Stokes system ($a=0$),
and it consists of the 2D heat kernel
\[
G(x,t)=\frac{1}{4\pi t}\, e^{-|x|^2/4t}
\]
and $2\times 2$ matrix
\begin{equation}
H(x,t)
=\int_t^\infty \nabla^2G(x,s)\,ds
=\int_t^\infty
G(x,s)\left(\frac{x\otimes x}{4s^2}-\frac{\mathbb I}{2s}\right)ds.
\label{funda-2}
\end{equation}
In 2D case one can write \eqref{funda-2} in terms of elementary functions
\begin{equation}
H(x,t)=
\frac{-(x\otimes x)}{|x|^2}\,G(x,t)
+\left(\frac{x\otimes x}{|x|^2}-\frac{\mathbb I}{2}\right)
\frac{1-e^{-|x|^2/4t}}{\pi|x|^2},
\label{funda-3}
\end{equation}
while one cannot in 3D, see \cite{FH1}.
One needs more careful argument than 3D case \cite{FH1}
to prove that \eqref{fundamental}
is actually the fundamental solution,
see Proposition \ref{justify}.

Indeed the integral representation \eqref{fundamental}
does not absolutely converge, but it is convergent due to
oscillation $O(at)^T$ with $a\in\mathbb R\setminus\{0\}$,
see Lemma \ref{converge}.
This is a contrast to the case $a=0$, in which \eqref{fundamental}
is not convergent.
In this case one needs the centering technique to recover the convergence,
which leads to the Stokes fundamental solution $E(x)$
given by \eqref{stokes-funda} as follows:
\begin{equation}
\int_0^\infty\left(K(x,t)-\frac{e^{-e/4t}}{8\pi t}\,\mathbb I\right)\,dt
=E(x).
\label{centering-st}
\end{equation}
This was clarified by 
Guenther and Thomann \cite[Proposition 2.2]{GT}.
As a part of this technique \eqref{centering-st},
the fundamental solution of the Laplace
operator in two dimensions is recovered exactly as
\begin{equation}
\int_0^\infty \left(G(x,t)-\frac{e^{-1/4t}}{4\pi t}\right)dt 
=\frac{1}{2\pi}\log \frac{1}{|x|}
\label{centering}
\end{equation}
in terms of the heat kernel,
see \cite[Lemma 2.1]{GT}.
Although we do not need the centering technique in the representation 
\eqref{fundamental} itself,
we will use this technique to justify some formulae in this section.
\begin{remark}
In \cite[p.301]{FHM}
Farwig, Hishida and M\"uller mentioned that the integral kernel
$\int_0^\infty O(t)^TG(O(t)x-y,t)\,dt$
should be modified to 
recover the convergence in two dimensions.
But this is redundant
as we will see in Lemma \ref{converge}
by making use of the oscillation.
\label{redundant}
\end{remark}

For convenience we will collect a few elementary foumulae,
which will be used several times.
We omit the proof that is nothing but integration by parts.
In the first assertion below
it is possible to derive even faster decay 
$r^{-\{2(m-1)+2k\}}$ for every $k\in\mathbb N$
by $k$-times integration by parts, but \eqref{osc-1} and \eqref{osc-2}
are enough for later use.
Note that they are not absolutely convergent for $m\leq 1$
(the only case we need is $m=1$).
\begin{lemma}
Let $r>0$.

\begin{enumerate}
\item
Let $a\in\mathbb R\setminus\{0\}$ and $m>0$.
Then
\begin{equation}
\left|\int_0^\infty e^{iat}e^{-r^2/t}\,\frac{dt}{t^m}\right|
\leq \frac{C}{|a|r^{2m}},
\label{osc-1}
\end{equation}
\begin{equation}
\left|\int_0^\infty e^{iat}\int_t^\infty e^{-r^2/s}\,\frac{ds}{s^{m+1}}\, dt\right|
\leq \frac{C}{|a|r^{2m}}, 
\label{osc-2}
\end{equation}
with some $C=C(m)>0$ independent of $r>0$ and $a\in\mathbb R\setminus\{0\}$,
where $i=\sqrt{-1}$.

\item
Let $m>1$.
Then

\begin{equation}
\int_0^\infty e^{-r^2/t}\,\frac{dt}{t^m}=\frac{\gamma(m-1)}{r^{2(m-1)}},
\label{exa-1}
\end{equation}
\begin{equation}
\int_0^\infty \int_t^\infty e^{-r^2/s}\,\frac{ds}{s^{m+1}}\, dt
=\frac{\gamma(m-1)}{r^{2(m-1)}},
\label{exa-2}
\end{equation}
where $\gamma(\cdot)$ denotes the Euler gamma function.
\end{enumerate}
\label{formula}
\end{lemma}

We begin with the following lemma, from which the function 
\eqref{fundamental} is well-defined.
\begin{lemma}
Let $a\in\mathbb R\setminus\{0\}$.
Then the integral $\Gamma_a(x,y)$ given by \eqref{fundamental}
converges for every
$(x,y)\in\mathbb R^2\times \mathbb R^2$ with $x\neq y$.
\label{converge}
\end{lemma}

\begin{proof}
We decompose $\Gamma_a(x,y)$ as
\[
\Gamma_a(x,y)=\Gamma_a^0(x,y)+\Gamma_a^1(x,y), \qquad
\Gamma_a^1(x,y)=\Gamma_a^{11}(x,y)+\Gamma_a^{12}(x,y)
\]
with
\begin{equation}
\begin{split}
\Gamma_a^0(x,y)
&=\int_0^\infty O(at)^T G(O(at)x-y, t)\, dt,  \\
\Gamma_a^{11}(x,y)
&=\int_0^\infty\!\!\! O(at)^T\!\!\! \int_t^\infty\!\!\! 
G(O(at)x-y, s)\,
\frac{(O(at)x-y)\otimes (O(at)x-y)}{4s^2}\,ds\, dt,  \\
\Gamma_a^{12}(x,y)
&=\int_0^\infty O(at)^T \int_t^\infty G(O(at)x-y, s)\,
\frac{-1}{2s}\,ds\, dt.
\end{split}
\label{split-fun}
\end{equation}
We start with the convergence of $\Gamma_a^0(x,y)$ by using
the centering technique as in \eqref{centering}.
By \eqref{osc-1} we know that
\begin{equation}
\left|\int_0^\infty O(at)^T e^{-1/4t}\,\frac{dt}{t}\right|\leq\frac{C}{|a|}.
\label{cent-0}
\end{equation}
Hence, it suffices to show the convergence of
\begin{equation}
\int_0^\infty O(at)^T\,
\left(e^{-|O(at)x-y|^2/4t}-e^{-1/4t}\right)\,\frac{dt}{t}.
\label{1st}
\end{equation}
As we will see, this is absolutely convergent.
For large $t$, we have
\begin{equation*}
\begin{split}
\int_1^\infty \left|e^{-|O(at)x-y|^2/4t}-e^{-1/4t} \right|\frac{dt}{t}
&\leq \int_1^\infty \big||O(at)x-y|^2-1\big|\,\frac{dt}{4t^2} \\
&\leq \frac{(|x|+|y|)^2+1}{4}.
\end{split}
\end{equation*}
For small $t$, we use the relation
\begin{equation}
|O(at)x-y|^2
=|x-y|^2+2at\big(\dot{O}(a\theta t)x\big)\cdot\left(O(a\theta t)x-y\right)
\label{small-t}
\end{equation}
for some $\theta=\theta(a,t,x,y)\in (0,1)$,
where
$\dot{O}(t)=\frac{d}{dt}O(t)$.
Then we have
\begin{equation}
\int_0^1 \left|e^{-|O(at)x-y|^2/4t}-e^{-1/4t}\right|\,\frac{dt}{t}
\leq
\int_0^1 \left(e^{-|x-y|^2/4t}\,e^{|a||x|(|x|+|y|)/2}+e^{-1/4t}\right)\frac{dt}{t}
\label{small-t-1}
\end{equation}
with
\begin{equation}
\begin{split}
&\quad \int_0^1 e^{-|x-y|^2/4t}\,\frac{dt}{t}
=\left(\int_0^1+\int_1^{1/|x-y|^2}\right) e^{-1/4t}\,\frac{dt}{t} \\
&\leq 4+\int_1^{1/|x-y|^2}\frac{dt}{t}
=4+2\log\frac{1}{|x-y|} \qquad (0<|x-y|<1),
\end{split}
\label{small-t-2}
\end{equation}
while
\[
\int_0^1 e^{-|x-y|^2/4t}\,\frac{dt}{t}
\leq 4 \qquad (|x-y|\geq 1).
\]
This concludes that \eqref{1st} is absolutely convergent.

The next integral $\Gamma_a^{11}(x,y)$ is
absolutely convergent without centering as above.
Given $(x,y)$ with $x\neq y$, there is
$\delta=\delta(a,x,y)>0$ such that
\begin{equation}
0<\frac{|x-y|^2}{2}
\leq |O(at)x-y|^2
\leq\frac{3|x-y|^2}{2}, \qquad 0\leq\forall\, t\leq\delta,
\label{near-0}
\end{equation}
on account of 
$\displaystyle{\lim_{t\to 0}|O(at)x-y|^2=|x-y|^2}$.
This together with \eqref{exa-2} implies that
\begin{equation}
\begin{split}
&\quad \int_0^\infty \int_t^\infty e^{-|O(at)x-y|^2/4s}
\,\frac{|O(at)x-y|^2}{s^3}\, ds\,dt  \\
&\leq\int_0^\delta\int_t^\infty e^{-|x-y|^2/8s}\,
\frac{3|x-y|^2}{2s^3}\,ds \,dt
+\int_\delta^\infty \int_t^\infty \frac{(|x|+|y|)^2}{s^3}\,ds \,dt \\
&\leq\frac{3|x-y|^2}{2}
\int_0^\infty \int_t^\infty e^{-|x-y|^2/8s}\,\frac{ds}{s^3}\,dt
+\frac{(|x|+|y|)^2}{2}\int_\delta^\infty \frac{dt}{t^2} \\
&=C+\frac{(|x|+|y|)^2}{2\delta}.
\end{split}
\label{2nd-conv}
\end{equation}

Finally, similarly to the argument of convergence of $\Gamma_a^0(x,y)$,
we can discuss $\Gamma_a^{12}(x,y)$.
From \eqref{osc-2} it follows that
\begin{equation}
\left|\int_0^\infty O(at)^T \int_t^\infty e^{-1/4s}\,\frac{ds}{s^2}\,dt\right|
\leq\frac{C}{|a|}.
\label{cent-12}
\end{equation}
It thus remains to show the convergence of
\begin{equation}
\int_0^\infty O(at)^T \int_t^\infty
\left(e^{-|O(at)x-y|^2/4s}-e^{-1/4s}\right)\,\frac{ds}{s^2}\,dt.
\label{3rd}
\end{equation}
For large $t$, we have
\begin{equation*}
\begin{split}
\int_1^\infty \int_t^\infty \left|e^{-|O(at)x-y|^2/4s}-e^{-1/4s}\right|
\frac{ds}{s^2}\,dt
&\leq\int_1^\infty \int_t^\infty
\big||O(at)x-y|^2-1\big|\,\frac{ds}{4s^3}\,dt  \\
&\leq\frac{(|x|+|y|)^2+1}{8}.
\end{split}
\end{equation*}
For small $t$, as in \eqref{small-t-1},
we use \eqref{small-t} to find
\begin{equation}
\begin{split}
&\quad \int_0^1 \int_t^\infty\left|e^{-|O(at)x-y|^2/4s}-e^{-1/4s}\right|
\frac{ds}{s^2}\,dt  \\
&\leq \int_0^1 \int_t^\infty
\left(e^{-|x-y|^2/4s}\, e^{|a||x|t(|x|+|y|)/2s}+e^{-1/4s}\right)\,
\frac{ds}{s^2}\,dt \\
&\leq  
e^{|a||x|(|x|+|y|)/2} 
\int_0^1 \int_t^\infty 
e^{-|x-y|^2/4s}\,\frac{ds}{s^2}\, dt +4
\end{split}
\label{small-t-3}
\end{equation}
with
\begin{equation}
\begin{split}
&\quad \int_0^1 \int_t^\infty
e^{-|x-y|^2/4s}\,\frac{ds}{s^2}\, dt
=\left(\int_0^1 +\int_1^{1/|x-y|^2}\right)
\int_t^\infty e^{-1/4s}\,\frac{ds}{s^2}\, dt \\
&\leq 4+4\int_1^{1/|x-y|^2}(1-e^{-1/4s})\,ds
\leq 4+\int_1^{1/|x-y|^2}\frac{ds}{s}  \\
&=4+2\log\frac{1}{|x-y|} \qquad (0<|x-y|<1),
\end{split}
\label{small-t-4}
\end{equation}
while
\[
\int_0^1 \int_t^\infty
e^{-|x-y|^2/4s}\,\frac{ds}{s^2}\, dt \leq 4 \qquad (|x-y|\geq 1).
\]
This implies the absolute convergence of \eqref{3rd}.
We have completed the proof.
\end{proof}

We have concentrated ourselves only on the convergence
of the integral \eqref{fundamental}.
So the estimates appeared in the proof above are
not related to the asymptotic behavior
with respect to $(x,y)$ at large distance,
which will be discussed in a different way in Proposition \ref{key},
but we have tried to derive the singular behavior for $|x-y|\to 0$
as less as possible,
see \eqref{small-t-2}, \eqref{2nd-conv} and
\eqref{small-t-4}.
This behavior should be logarithmic,
otherwise \eqref{fundamental} cannot be the fundamental solution,
but the behavior \eqref{2nd-conv} is not clear
since $\delta$ depends on $x, y$
(probably, the part $\Gamma_a^{11}(x,y)$ would be bounded 
for $|x-y|\to 0$ as in the second term of the Stokes fundamental solution
\eqref{stokes-funda}).
In order to ensure that the volume potential \eqref{vol} below
is well-defined,
we will show the following lemma.
The growth rate \eqref{loc-sum-y} with $\rho=2|x|$ will be also used
to show asymptotic representation \eqref{asy-u} below.
\begin{lemma}
Let $a\in \mathbb R\setminus\{0\}$
There is a constane $C>0$ independent of $a\in \mathbb R\setminus\{0\}$
such that
\begin{equation}
\int_{|y|\leq\rho}|\Gamma_a(x,y)|\,dy
\leq C|a|^{-1}\rho^2+C\rho^2\log\rho,
\label{loc-sum-y}
\end{equation}
\begin{equation}
\int_{|y|\leq\rho}|\nabla_x\Gamma_a(x,y)|\,dy
\leq C\rho,
\label{loc-sum-grad}
\end{equation}
for every $x\in\mathbb R^2$ and $\rho\geq |x|+e$.
\label{loc-sum}
\end{lemma}

\begin{proof}
To this end, it is convenient to
use another representation \eqref{funda-3} of $H(x,t)$ together with
the centering technique \eqref{centering-st} due to 
Guenther and Thomann \cite{GT}.
But we subtract $\left(e^{-1/4t}/8\pi t\right)\mathbb I$
instead of $\left(e^{-e/4t}/8\pi t\right)\mathbb I$
since there is no need to derive $E(x)$.
First of all, it follows from \eqref{osc-1} that
\begin{equation}
\left|\int_0^\infty O(at)^T\,\frac{e^{-1/4t}}{8\pi t}\,dt\right|
\leq\frac{C}{|a|}.
\label{subt}
\end{equation}
We set
\begin{equation}
\widetilde\Gamma_a(x,y)
:=\int_0^\infty O(at)^T
\left(K(O(at)x-y,t)-\frac{e^{-1/4t}}{8\pi t}\mathbb I \right)\,dt.
\label{tilde}
\end{equation}
We will see that the integral of this integrand
over $(0,\infty)\times \overline{B_\rho(0)}$ with respect to $(t,y)$
is absolutely convergent.
By the transformation $y=O(at)z$ we have
\begin{equation}
\begin{split}
&\quad \int_0^\infty\int_{|y|\leq\rho}
\left|K(O(at)x-y, t)-\frac{e^{-1/4t}}{8\pi t}\mathbb I\right|\,dy\,dt  \\
&=\int_0^\infty\int_{|z|\leq\rho}
\left|K(x-z,t)-\frac{e^{-1/4t}}{8\pi t}\mathbb I\right|\,dz\,dt.
\end{split}
\label{trans}
\end{equation}
The useful decomposition discovered by \cite{GT} is
\begin{equation}
\begin{split}
&\quad K(x,t)-\frac{e^{-1/4t}}{8\pi t}\mathbb I \\
&=\frac{e^{-|x|^2/4t}-e^{-1/4t}}{8\pi t}\mathbb I
+\left(\frac{e^{-|x|^2/4t}}{8\pi t}
-\frac{1-e^{-|x|^2/4t}}{2\pi |x|^2}
\right)
\left(\mathbb I-\frac{2x\otimes x}{|x|^2}\right)
=:\frac{A+B}{8\pi}.
\end{split}
\label{cen-sp}
\end{equation}
Then we find
\begin{equation}
\int_0^\infty |A|\,dt
=C\,\big|\log |x|\big|,
\label{cen-1}
\end{equation}
while we see from 
the transformation $\tau=|x|^2/4t$ that
\begin{equation}
\int_0^\infty |B|\,dt
=\left|\mathbb I-\frac{2x\otimes x}{|x|^2}\right|
\int_0^\infty \frac{-\tau e^{-\tau}+1-e^{-\tau}}{\tau^2}\,d\tau
=\left|\mathbb I-\frac{2x\otimes x}{|x|^2}\right| 
\leq C
\label{cen-2}
\end{equation}
since
\[
0<\frac{-\tau e^{-\tau}+1-e^{-\tau}}{\tau^2}
=\frac{d}{d\tau}\left(\frac{e^{-\tau}-1}{\tau}\right)
\]
for every $\tau >0$.
By the Fubini theorem we obtain
\begin{equation*}
\begin{split}
\int_{|y|\leq\rho}|\widetilde\Gamma_a(x,y)|\,dy
&\leq C\int_{|y|\leq\rho}\left(1+\big|\log |x-y|\big|\right)\,dy \\
&\leq C\rho^2+C\int_{|y-x|\leq\rho+|x|}\big|\log |x-y|\big|\,dy \\
&\leq C\rho^2+C\rho^2\log\rho
\end{split}
\end{equation*}
for $\rho\geq |x|+e$.
This together with \eqref{subt} concludes \eqref{loc-sum-y}.

For the estimate of $\nabla_x\Gamma_a(x,y)$,
we first need to justify
\begin{equation}
\nabla_x\Gamma_a(x,y)
=\int_0^\infty O(at)^T\nabla_x \big[K(O(at)x-y,t)\big]\,dt
\label{grad-funda}
\end{equation}
with the aid of 
$\widetilde \Gamma_a(x,y)$ given by \eqref{tilde}.
Given $\varphi\in C_0^\infty(\mathbb R^2)$ arbitrarily, we have
\[
\langle\widetilde\Gamma_a(\cdot,y), \mbox{div $\varphi$}\rangle
=\int_0^\infty \left\langle O(at)^T
\left(K(O(at)x-y,t)-\frac{e^{-1/4t}}{8\pi t}\mathbb I\right),\,
\mbox{div $\varphi$}\right\rangle\,dt
\]
because this integral
over $(0,\infty)\times B_L(0)$ with respect to $(t,x)$
is absolutely convergent
by the same reasoning as in the proof of \eqref{loc-sum-y},
where $L>0$ is taken in such a way that
$\mbox{Supp $\varphi$}\subset B_L(0)$.
We then use
\begin{equation}
|(\nabla K)(x,t)|
\leq Ct^{-3/2}e^{-|x|^2/16t}
+C\int_t^\infty s^{-5/2}e^{-|x|^2/16s}\,ds
\label{K-grad}
\end{equation}
together with \eqref{exa-1}--\eqref{exa-2} to get the absolute convergence
\begin{equation*}
\begin{split}
\int_0^\infty\int_{|x|\leq L}
\left|\nabla_x\big[K(O(at)x-y,t)\big]\right|\,dx\,dt
&\leq C\int_0^\infty\int_{|x|\leq L}
|(\nabla K)(O(at)x-y,t)|\,dx\,dt  \\
&=C\int_0^\infty\int_{|x|\leq L}
|(\nabla K)(x-y,t)|\,dx\,dt  \\
&\leq C\int_{|x|\leq L}\frac{dx}{|x-y|}
\end{split}
\end{equation*}
as in \eqref{trans}.
Hence we obtain
\begin{equation*}
\begin{split}
\langle\widetilde\Gamma_a(\cdot,y), \mbox{div $\varphi$}\rangle
&=-\int_0^\infty\left\langle O(at)^T\nabla_x\big[K(O(at)x-y,t)\big], 
\varphi\right\rangle\,dt  \\
&=-\left\langle \int_0^\infty O(at)^T\nabla_x\big[K(O(at)x-y,t)\big]\,dt,
\varphi\right\rangle
\end{split}
\end{equation*}
for all $\varphi\in C_0^\infty(\mathbb R^2)$,
which implies \eqref{grad-funda} since
$\nabla_x\widetilde\Gamma_a(x,y)=\nabla_x\Gamma_a(x,y)$.
Once we have that, by the same reasoning as above we get
\[
\int_{|y|\leq\rho}|\nabla_x\Gamma_a(x,y)|\,dy
\leq C\int_{|y|\leq\rho}\int_0^\infty
|(\nabla K)(O(at)x-y,t)|\,dt\,dy
\leq C\int_{|y|\leq\rho}\frac{dy}{|x-y|}
\]
which leads to \eqref{loc-sum-grad} for $\rho\geq |x|+e$.
\end{proof}

The following estimate provides the decay structure of $\Gamma_a(x,y)$
and plays a crucial role in this paper.
\begin{proposition}
Let $a\in\mathbb R\setminus\{0\}$.
\begin{enumerate}
\item
There is a constant $C>0$ independent of $a\in\mathbb R\setminus\{0\}$
such that
\begin{equation}
\left|\Gamma_a(x,y)-\frac{x^\perp\otimes y^\perp}{4\pi |x|^2}\right|
\leq \frac{C(|a|^{-1}+|y|^2)}{|x|^2}
\label{asy-xy}
\end{equation}
for all $(x,y)\in\mathbb R^2\times \mathbb R^2$ with $|x|> 2|y|$.
In particular, we have
\begin{equation}
\Gamma_a(x,y)
=\frac{x^\perp\otimes y^\perp}{4\pi |x|^2}
+O(|x|^{-2}),
\label{asym-funda}
\end{equation}
as $|x|\to\infty$ so long as $|y|\leq\rho$,
where $\rho >0$ is fixed.

\item
Similarly, there is a constant $C>0$ independent of $a\in\mathbb R\setminus\{0\}$
such that
\begin{equation}
\left|\Gamma_a(x,y)-\frac{x^\perp\otimes y^\perp}{4\pi |y|^2}\right| 
\leq \frac{C(|a|^{-1}+|x|^2)}{|y|^2}
\label{asy-xy2}
\end{equation}
for all $(x,y)\in\mathbb R^2\times \mathbb R^2$ with $|y|> 2|x|$.
\end{enumerate}
\label{key}
\end{proposition}

\begin{proof}
Tha latter assertion
follows from the former one because 
$\Gamma_a(x,y)=\!\Gamma_{-a}(y,x)^T$ and
$(y^\perp\otimes x^\perp)^T=x^\perp\otimes y^\perp$.
We will show \eqref{asy-xy}, which immediately implies \eqref{asym-funda}.
Let us start with $\Gamma_a^0(x,y)$ given by \eqref{split-fun}.
We use the Taylor formula with respect to $y$
around $y=0$ to see that
there is $\theta=\theta(a,t,x,y)\in (0,1)$ satisfying
\begin{equation}
\begin{split}
&\quad  e^{-|O(at)x-y|^2/4t}  \\
&=e^{-|x|^2/4t}+e^{-|x|^2/4t}\,\frac{\left(O(at)x\right)\cdot y}{2t}  \\
&\qquad +\frac{1}{2}\,e^{-|O(at)x-\theta y|^2/4t}\; y^T\,
\frac{\left(O(at)x-\theta y\right)\otimes \left(O(at)x-\theta y\right)-2t\,\mathbb I}
{4t^2}\, y.
\end{split}
\label{taylor}
\end{equation}
According to this formula, we decompose $\Gamma_a^0(x,y)$ as
\[
\Gamma_a^0(x,y)
=\Gamma_a^{01}(x)+\Gamma_a^{02}(x,y)+\Gamma_a^{03}(x,y).
\]
It follows from \eqref{osc-1} that
\begin{equation}
|\Gamma_a^{01}(x)|
=\left| \frac{1}{4\pi}\int_0^\infty O(at)^T e^{-|x|^2/4t}\,\frac{dt}{t}\right|
\leq \frac{C}{|a||x|^2}.
\label{decay-01}
\end{equation}
Since
\begin{equation}
(O(at)x)\cdot y=(x\cdot y)\cos at+(x^\perp\cdot y)\sin at
\label{from2nd}
\end{equation}
and, thereby,
\begin{equation}
\begin{split}
\left\{(O(at)x)\cdot y\right\}O(at)^T
&=\frac{1}{2}
\left( 
\begin{array}{cc}
x\cdot y & x^\perp\cdot y \\
-x^\perp\cdot y & x\cdot y
\end{array} 
\right)
+\frac{\cos 2at}{2}
\left(
\begin{array}{cc}
x\cdot y & -x^\perp\cdot y \\
x^\perp\cdot y & x\cdot y
\end{array}
\right)  \\
&\qquad  +\frac{\sin 2at}{2}
\left(
\begin{array}{cc} 
x^\perp\cdot y & x\cdot y \\
-x\cdot y & x^\perp\cdot y
\end{array}
\right),
\end{split}
\label{from-2}
\end{equation}
we have
\begin{equation}
\begin{split}
\Gamma_a^{02}(x,y)
&=
\frac{1}{16\pi}\int_0^\infty e^{-|x|^2/4t}\,\frac{dt}{t^2}
\left(
\begin{array}{cc}
x\cdot y & x^\perp\cdot y \\
-x^\perp\cdot y & x\cdot y
\end{array}
\right)
+M_a^{02}(x,y)  \\
&=\frac{1}{4\pi |x|^2}
\left( 
\begin{array}{cc}
x\cdot y & x^\perp\cdot y \\
-x^\perp\cdot y & x\cdot y
\end{array}
\right)
+M_a^{02}(x,y)
\end{split}
\label{decay-02-L}
\end{equation}
with
\begin{equation}
|M_a^{02}(x,y)|
\leq\frac{C|y|}{|a||x|^3}
\leq\frac{C}{|a||x|^2}
\label{decay-02}
\end{equation}
for $|x|>2|y|$,
which follows from \eqref{osc-1}.
Since
$e^{-|O(at)x-\theta y|^2/4t}\leq e^{-|x|^2/16t}$
for $|y|<|x|/2$,
it is easily seen that
\begin{equation}
|\Gamma_a^{03}(x,y)|
\leq C|y|^2\int_0^\infty \left( |x|^2t^{-3}+t^{-2}\right)e^{-|x|^2/16t}\,dt
= \frac{C|y|^2}{|x|^2}
\label{decay-03}
\end{equation}
without using oscillation.
Then \eqref{decay-01}, \eqref{decay-02-L},
\eqref{decay-02} and \eqref{decay-03} imply that
\begin{equation}
\left|
\Gamma_a^0(x,y)
-\frac{1}{4\pi |x|^2}
\left(
\begin{array}{cc}
x\cdot y & x^\perp\cdot y \\
-x^\perp\cdot y & x\cdot y
\end{array}
\right)
\right|
\leq C(|a|^{-1}+|y|^2)|x|^{-2}
\label{decay-0}
\end{equation}
for $|x|> 2|y|$.

We proceed to the decay structure of $\Gamma_a^1(x,y)$
given by \eqref{split-fun}.
Similarly to \eqref{taylor}, we have the formula
\begin{equation}
\begin{split}
&\quad  e^{-|O(at)x-y|^2/4s}  \\
&=e^{-|x|^2/4s}+e^{-|x|^2/4s}\,\frac{\left(O(at)x\right)\cdot y}{2s}  \\
&\qquad +\frac{1}{2}\,e^{-|O(at)x-\theta y|^2/4s}\; y^T\,
\frac{\left(O(at)x-\theta y\right)\otimes \left(O(at)x-\theta y\right)-2s\,\mathbb I}
{4s^2}\, y
\end{split}
\label{taylor2}
\end{equation}
with some $\theta=\theta(a,t,s,x,y)\in (0,1)$ and, correspondingly,
we decompose $\Gamma_a^{11}(x,y)$ given by \eqref{split-fun} as
\[
\Gamma_a^{11}(x,y)=\Gamma_a^{111}(x,y)+\Gamma_a^{112}(x,y)+\Gamma_a^{113}(x,y).
\]
We write
\begin{equation}
\begin{split}
&\quad O(at)^T [(O(at)x-y)\otimes (O(at)x-y)] \\
&=(x-O(at)^Ty)\otimes (O(at)x-y) \\
&=A_0+(\cos at)\,A_c+(\sin at)\,A_s
+\frac{\cos 2at}{2}\,\widetilde A_c
+\frac{\sin 2at}{2}\,\widetilde A_s
\end{split}
\label{A}
\end{equation}
with
\begin{equation*}
\begin{split}
&A_0=A_0(x,y)
=\frac{-3(x\otimes y)+(x^\perp\otimes y^\perp)}{2}, \\
&A_c=A_c(x,y)
=\left(
\begin{array}{cc}
x_1^2+y_1^2 & x_1x_2+y_1y_2 \\
x_1x_2+y_1y_2 & x_2^2+y_2^2 
\end{array}
\right),  \\
&A_s=A_s(x,y)
=\left(
\begin{array}{cc}
-x_1x_2+y_1y_2 & x_1^2+y_2^2 \\
-(x_2^2+y_1^2) & x_1x_2-y_1y_2 
\end{array}
\right), \\
&\widetilde A_c=\widetilde A_c(x,y)
=\left(
\begin{array}{cc}
-x\cdot y & x^\perp\cdot y \\
-x^\perp\cdot y & -x\cdot y
\end{array}
\right),   \\
&\widetilde A_s=\widetilde A_s(x,y)
=\left(
\begin{array}{cc} 
-x^\perp\cdot y & -x\cdot y \\
x\cdot y & -x^\perp\cdot y
\end{array}
\right).
\end{split}
\end{equation*}
Using \eqref{osc-2} and \eqref{exa-2}, we get
\begin{equation}
\begin{split}
\Gamma_a^{111}(x,y)
&=\frac{A_0}{16\pi}\int_0^\infty \int_t^\infty
e^{-|x|^2/4s}\,\frac{ds}{s^3}\,dt
+M_a^{111}(x,y)  \\
&=
\frac{-3(x\otimes y)+(x^\perp\otimes y^\perp)}{8\pi |x|^2}
+M_a^{111}(x,y)
\end{split}
\label{decay-11-L}
\end{equation}
with
\begin{equation}
|M_a^{111}(x,y)|\leq \frac{C}{|a||x|^2}
\label{decay-11}
\end{equation}
for $|x|>2|y|$.
Look at \eqref{from2nd} and \eqref{A} to obtain
\[
\{(O(at)x)\cdot y\}
O(at)^T [(O(at)x-y)\otimes (O(at)x-y)]
=B_0+(\mbox{remainder})
\]
with
\[
B_0=\frac{x\cdot y}{2}A_c+\frac{x^\perp\cdot y}{2}A_s 
=\frac{x\cdot y}{2}(x\otimes x)+\frac{x^\perp\cdot y}{2}(x\otimes x^\perp)
+B_1
=\frac{|x|^2(x\otimes y)}{2}+B_1
\]
which is independent of $t$,
where $B_1$ is of degree one (resp. three) 
with respect to $x$ (resp. $y$)
and the remainder consists of all terms involving
$\cos kat$ and $\sin kat$ ($1\leq k\leq 3$).
We thus find from \eqref{osc-2} and \eqref{exa-2} that
\begin{equation}
\begin{split}
\Gamma_a^{112}(x,y)
&=\frac{|x|^2(x\otimes y)}{64\pi}
\int_0^\infty \int_t^\infty e^{-|x|^2/4s}\,\frac{ds}{s^4}\, dt
+M^{112}_a(x,y)  \\
&=\frac{x\otimes y}{4\pi|x|^2}
+M^{112}_a(x,y)
\end{split}
\label{decay-12-L}
\end{equation}
with
\begin{equation}
|M^{112}_a(x,y)|
\leq \frac{C\,(|a|^{-1}|y|+|y|^3)}{|x|^3}
\leq \frac{C\,(|a|^{-1}+|y|^2)}{|x|^2}
\label{decay-12}
\end{equation}
for $|x|>2|y|$.
Without using oscillation, we see that
\begin{equation}
|\Gamma_a^{113}(x,y)|
\leq C|y|^2|x|^2\int_0^\infty \int_t^\infty
\left(|x|^2s^{-5}+s^{-4}\right) e^{-|x|^2/16s}\,ds\, dt
= \frac{C|y|^2}{|x|^2}
\label{decay-13}
\end{equation}
for $|x|>2|y|$.
We collect \eqref{decay-11-L}, \eqref{decay-11}, 
\eqref{decay-12-L}, \eqref{decay-12} and \eqref{decay-13} to find
\begin{equation}
\left|
\Gamma_a^{11}(x,y)
-\frac{-(x\otimes y)+(x^\perp\otimes y^\perp)}{8\pi |x|^2}
\right|
\leq C(|a|^{-1}+|y|^2)|x|^{-2}
\label{decay-1}
\end{equation}
for $|x|>2|y|$.

Finally, we decompose $\Gamma_a^{12}(x,y)$ given by \eqref{split-fun} as
\[
\Gamma_a^{12}(x,y)
=\Gamma_a^{121}(x)+\Gamma_a^{122}(x,y)+\Gamma_a^{123}(x,y)
\]
by use of \eqref{taylor2} and deduce its decay structure.
By \eqref{osc-2} we have
\begin{equation}
|\Gamma_a^{121}(x)|\leq \frac{C}{|a||x|^2}.
\label{decay-21}
\end{equation}
As in the argument for $\Gamma_a^{02}(x,y)$,
we employ \eqref{from-2} to obtain
\begin{equation}
\begin{split}
\Gamma_a^{122}(x,y)
&=
\frac{-1}{32\pi}\int_0^\infty \int_t^\infty
e^{-|x|^2/4s}\,\frac{ds}{s^3}\, dt
\left(
\begin{array}{cc}
x\cdot y & x^\perp\cdot y \\
-x^\perp\cdot y & x\cdot y
\end{array}
\right)
+M^{122}_a(x,y)  \\
&=\frac{-1}{8\pi |x|^2}
\left(
\begin{array}{cc}
x\cdot y & x^\perp\cdot y \\
-x^\perp\cdot y & x\cdot y
\end{array}
\right)
+M^{122}_a(x,y)
\end{split}
\label{decay-22-L}
\end{equation}
with
\begin{equation}
|M^{122}_a(x,y)|
\leq\frac{C|y|}{|a||x|^3}
\leq\frac{C}{|a||x|^2}
\label{decay-22}
\end{equation}
for $|x|>2|y|$.
Similarly to the argument for $\Gamma_a^{113}(x,y)$, it is seen that
\begin{equation}
|\Gamma_a^{123}(x,y)|
\leq C|y|^2\int_0^\infty \int_t^\infty
\left(|x|^2s^{-4}+s^{-3}\right)
e^{-|x|^2/16s}\,ds\, dt
= \frac{C|y|^2}{|x|^2}
\label{decay-23}
\end{equation}
for $|x|>2|y|$.
We collect \eqref{decay-21}, \eqref{decay-22-L}, \eqref{decay-22} and
\eqref{decay-23} to obtain
\begin{equation}
\left|
\Gamma_a^{12}(x,y)
-\frac{-1}{8\pi|x|^2}
\left(
\begin{array}{cc}
x\cdot y & x^\perp\cdot y \\
-x^\perp\cdot y & x\cdot y
\end{array}
\right)
\right|
\leq C(|a|^{-1}+|y|^2)|x|^{-2}
\label{decay-2}
\end{equation}
for $|x|>2|y|$.
Using the simple relation
\[
\left(
\begin{array}{cc}
x\cdot y & x^\perp\cdot y \\
-x^\perp\cdot y & x\cdot y
\end{array}
\right)
=x\otimes y+x^\perp\otimes y^\perp,
\]
we gather \eqref{decay-0}, \eqref{decay-1} and \eqref{decay-2}
to conclude \eqref{asym-funda}.
The proof is complete.
\end{proof}

We next verify that \eqref{fundamental} can be actually the fundamental solution.
To this end, we need two lemmas.
\begin{lemma}
Let $f\in L^1(\mathbb R^2)\cap L^\infty(\mathbb R^2)$ and
\begin{equation}
p(x)=\int_{\mathbb R^2}Q(x-y)\cdot f(y)\,dy,
\label{vol-p}
\end{equation}
where $Q(x)$ is given by \eqref{pre-funda}.
Set 
\begin{equation}
v^0(x,t)
=O(at)^T \int_{\mathbb R^2}G(O(at)x-y, t)f(y)\, dy,
\label{evolu-0}
\end{equation}
\begin{equation}
v^1(x,t)
=O(at)^T \int_{\mathbb R^2}H(O(at)x-y, t)f(y)\, dy,
\label{evolu-1}
\end{equation}
where $H(x,t)$ is given by \eqref{funda-2}.
Then they respectively satisfy
\begin{equation}
\partial_tv^0+L_a v^0=0, \qquad v^0(\cdot,0)=f,
\label{cp-0}
\end{equation}
\begin{equation}
\partial_tv^1+L_a v^1=0, \qquad v^1(\cdot,0)=-\nabla p,
\label{cp-1}
\end{equation}
in $\mathbb R^2\times (0,\infty)$, where
\begin{equation}
L_a v:=-\Delta v-a\left(x^\perp\cdot\nabla v-v^\perp\right).
\label{diff-op}
\end{equation}
\label{unsteady}
\end{lemma}

\begin{proof}
The well-known estimate of singular integrals yields
$\nabla p\in L^q(\mathbb R^2)$ for every $q\in (1,\infty)$.
By the derivation \eqref{transform} of the equation \eqref{rot-ns},
it is obvious that $v^0(x,t)$ is a solution to
the Cauchy problem \eqref{cp-0},
where the initial condition is understood as
$\lim_{t\to 0}\|v^0(t)-f\|_{L^q(\mathbb R^2)}=0$
for every $q\in (1,\infty)$.
By the same reasoning,
$v^1(x,t)=(v_1^1,v_2^1)^T$ with
\begin{equation*}
\begin{split}
v^1_j(x,t)
&=-\sum_k O(at)_{kj} \int_{\mathbb R^2}G(O(at)x-y, t)\partial_k p(y)\, dy  \\
&=-\int_{\mathbb R^2}\partial_{x_j}G(O(at)x-y, t)
\int_{\mathbb R^2}Q(y-z)\cdot f(z)\,dz\, dy  \qquad (j=1,2)
\end{split}
\end{equation*}
solves \eqref{cp-1}.
Note that the integration by parts above can be justified
since $p\in L^r(\mathbb R^2)$ for every $r\in (2,\infty)$
by the Hardy-Littlewood-Sobolev inequality.
So we have only to deduce the representation \eqref{evolu-1}.
Using the relation
\[
Q(y)=\frac{y}{2\pi |y|^2}=-\int_0^\infty \nabla G(y,\tau)\, d\tau
\]
and the semigroup property of the heat kernel, we find
\begin{equation*}
\begin{split}
v^1_j(x,t)
&=\int_{\mathbb R^2}\sum_m 
\int_0^\infty \int_{\mathbb R^2}
\partial_{x_j}G(O(at)x-y, t)(\partial_m G)(y-z,\tau)\,dy\,
d\tau\, f_m(z)\,dz  \\
&=-\int_{\mathbb R^2}\sum_m
\int_0^\infty \partial_{x_j}\partial_{z_m} G(O(at)x-z,t+\tau)\,d\tau\, f_m(z)\,dz  \\
&=\int_{\mathbb R^2}\sum_{k,m}O(at)_{kj}\int_t^\infty
(\partial_k \partial_mG)(O(at)x-z,s)\,ds\, f_m(z)\, dz
\end{split}
\end{equation*}
which leads us to \eqref{evolu-1}.
\end{proof}

\begin{lemma}
Let $\varepsilon \geq 0$.
Let $U\in{\cal S}^\prime(\mathbb R^2)$ fulfill
\[
\varepsilon U-\Delta U-a\,x^\perp\cdot\nabla U=0 \qquad\mbox{in $\mathbb R^2$},
\]
where ${\cal S}^\prime$ is the class of tempered distributions.
Then $\mbox{\emph{Supp} $\widehat U$}\subset\{0\}$,
where $\widehat U$ denotes the Fourier transform of $U$.
Similarly, if $u\in {\cal S}^\prime(\mathbb R^2)$ and
$p\in{\cal S}^\prime(\mathbb R^2)$ satisfy
\begin{equation}
\varepsilon u-\Delta u-a\big(x^\perp\cdot\nabla u-u^\perp\big)+\nabla p=0, \qquad
\mbox{\emph{div} $u$}=0 \qquad\mbox{in $\mathbb R^2$},
\label{rot-homo}
\end{equation}
then
$\mbox{\emph{Supp} $\widehat u$}\subset\{0\}$ and
$\mbox{\emph{Supp} $\widehat p$}\subset\{0\}$.
\label{f-supp}
\end{lemma}

\begin{proof}
We will prove the second assertion along the same idea
as in \cite{FHM}, \cite[Lemma 4.2]{H06}
(in which the first assertion was shown for the case
$\varepsilon =0$).
By \eqref{solenoidal} we have
$\Delta p=0$, so that $\mbox{Supp $\widehat p$}\subset\{0\}$
is obvious.
We take the Fourier transform of
$\mbox{\eqref{rot-homo}}_1$ to find
\[
(\varepsilon +|\xi|^2)\,\widehat u
-a\,\left(\xi^\perp\cdot\nabla_\xi\widehat u-\widehat u^\perp\right)
+i\xi\widehat p=0.
\]
Given $\psi\in C_0^\infty(\mathbb R^2\setminus\{0\})$ arbitrarily,
we set
\[
\phi(\xi)
=\int_0^\infty O(at)\,e^{-(\varepsilon +|\xi|^2)t}\psi(O(at)^T\xi)\,dt
\in C_0^\infty(\mathbb R^2\setminus\{0\}),
\]
which solves
\[
(\varepsilon +|\xi|^2)\,\phi
+a\,\left(\xi^\perp\cdot\nabla_\xi\phi-\phi^\perp\right)=\psi.
\]
We thus obtain
\begin{equation*}
\begin{split}
\langle\widehat u, \psi\rangle
&=\left\langle \widehat u,
(\varepsilon +|\xi|^2)\,\phi
+a\,\left(\xi^\perp\cdot\nabla_\xi\phi-\phi^\perp\right)\right\rangle \\
&=\left\langle 
(\varepsilon +|\xi|^2)\,\widehat u
-a\,\left(\xi^\perp\cdot\nabla_\xi\widehat u-\widehat u^\perp\right),
\phi\right\rangle  \\
&=-\langle i\xi\widehat p, \phi\rangle=0,
\end{split}
\end{equation*}
which completes the proof.
\end{proof}

The following volume potential \eqref{vol} is well-defined 
on account of \eqref{loc-sum-y}
and provides a solution to \eqref{rot-st-wh} for every
$f\in C_0^\infty(\mathbb R^2)$;
that is, $\Gamma_a(x,y)$ is a fundamental solution.
We also deduce several properties of \eqref{vol} for later use, including 
asymptotic representation \eqref{asy-u}
even for less regular $f$, whose support is not necessarily compact but
which decays sufficiently fast at infinity.
\begin{proposition}
Let $a\in \mathbb R\setminus\{0\}$.
Suppose
\begin{equation}
f\in L^1(\mathbb R^2)\cap L^\infty(\mathbb R^2).
\label{f-basic}
\end{equation}
Set
\begin{equation}
u(x)=\int_{\mathbb R^2}\Gamma_a(x,y)f(y)\,dy,
\label{vol}
\end{equation}
where $\Gamma_a(x,y)$ is given by \eqref{fundamental}, and consider
$p(x)$ defined by \eqref{vol-p} as well.

\begin{enumerate}
\item
The function $u(x)$ is well-defined by \eqref{vol}
as an element of $L^\infty_{loc}(\mathbb R^2)\cap {\cal S}^\prime(\mathbb R^2)$.

\item
Suppose further that
\begin{equation}
\int_{\mathbb R^2}|x||f(x)|\,dx <\infty, \qquad
f(x)=O(|x|^{-3}(\log |x|)^{-1}) \quad\mbox{as $|x|\to\infty$}.
\label{f-cond-1}
\end{equation}
Then the functions $u(x)$ and $p(x)$ enjoy
\begin{equation}
|u(x)|+|\nabla u(x)|+|p(x)|=O(|x|^{-1}) \quad\mbox{as $|x|\to\infty$}
\label{decay-rate}
\end{equation}
with estimate
\begin{equation}
\sup_{|x|\geq \rho}|x||u(x)|
\leq C(1+|a|^{-1})\left[\int_{\mathbb R^2}(1+|x|)|f(x)|\,dx
+\sup_{|x|\geq \rho/2}|x|^3(\log |x|)|f(x)|\right]
\label{est-u-far}
\end{equation}
for every $\rho \geq e$, 
where the constant $C>0$ is independent of $\rho\in [e,\infty)$ and
$a\in\mathbb R\setminus\{0\}$.
Furthermore, we have
\begin{equation}
p(x)
=\int_{\mathbb R^2}f\,dy \cdot\frac{x}{2\pi|x|^2}+O(|x|^{-2})
\quad\mbox{as $|x|\to\infty$}.
\label{asy-p}
\end{equation}

\item
In addition to \eqref{f-basic} and \eqref{f-cond-1}, assume \eqref{f-cond-2}.
Then we have
\begin{equation}
u(x)
=\int_{\mathbb R^2}y^\perp\cdot f\,dy\; \frac{x^\perp}{4\pi|x|^2}
+ (1+|a|^{-1})\,o(|x|^{-1})
\quad\mbox{as $|x|\to\infty$}.
\label{asy-u}
\end{equation}
If in particular the support of $f$ is compact,
then the remainder decays like $O(|x|^{-2})$ in \eqref{asy-u}.

\item 
Under the conditions \eqref{f-basic} and \eqref{f-cond-1},
the pair $\{u,p\}$ satisfies
\begin{equation}
-\Delta u-a\left(x^\perp\cdot\nabla u-u^\perp\right)+\nabla p=f, \qquad
\mbox{\emph{div} $u$}=0 \qquad\mbox{in $\mathbb R^2$}
\label{rot-st-wh}
\end{equation}
in the sense of distributions as well as
\begin{equation}
(\nabla^2u,\; \nabla p,\; x^\perp\cdot\nabla u-u^\perp)
\in L^q(\mathbb R^2) \qquad\mbox{for $\forall\, q\in (1,\infty)$},
\label{sum-property1}
\end{equation}
\begin{equation}
x^\perp\cdot\nabla u\in L^r(\mathbb R^2)
\qquad\mbox{for $\forall\, r\in (2,\infty)$}.
\label{sum-property2}
\end{equation}
If in addition $f\in C^\infty(\mathbb R^2)$, then we have
$\{u,p\}\in C^\infty(\mathbb R^2)$.

\end{enumerate}
\label{justify}
\end{proposition}
\begin{remark}
It is also possible to deduce
$\nabla u(x)=O(|x|^{-2})$ at infinity
by use of similar estimates of $\nabla_x\Gamma_a(x,y)$,
see \eqref{grad-funda}, to
Proposition \ref{key} (such estimates of $\nabla_x\Gamma_a(x,y)$ are not
simple consequences of Proposition \ref{key} and
one needs further several pages).
Since slower decay $\nabla u(x)=O(|x|^{-1})$ 
in \eqref{decay-rate} is enough for the proof of Theorem \ref{decay},
we postpone precise analysis of 
$\nabla_x\Gamma_a(x,y)$ until a forthcoming paper,
in which the external force $f=\mbox{\emph{div} $F$}$ with
$F(x)=O(|x|^{-2})$ will be treated by using estimates of
$\nabla_y\Gamma_a(x,y)$.
\label{rem-grad}
\end{remark}

\noindent
{\it Proof of Proposition \ref{justify}}.
Let $|x|\geq e$, then we take $\rho=2|x|$ in \eqref{loc-sum-y}
to obtain
\[
\int_{|y|\leq 2|x|}|\Gamma_a(x,y)||f(y)|\,dy
\leq C(1+|a|^{-1})\|f\|_{L^\infty(\mathbb R^2)}|x|^2\log |x| \qquad (|x|\geq e).
\]
By \eqref{asy-xy2} we also have
\begin{equation*}
\begin{split}
\int_{|y|>2|x|}|\Gamma_a(x,y)||f(y)|\,dy
&\leq C\int_{|y|>2|x|}\left(\frac{|x|}{|y|}+\frac{1}{|a||y|^2}\right)|f(y)|\,dy \\
&\leq C(1+|a|^{-1})\|f\|_{L^1(\mathbb R^2)}  \qquad (|x|\geq e).
\end{split}
\end{equation*}
When $|x|<e$, we similarly use \eqref{loc-sum-y} with $\rho=2e$ and
\eqref{asy-xy2} to find
\begin{equation}
|u(x)|
\leq\int_{|y|\leq 2e}+\int_{|y|>2e}
\leq C(1+|a|^{-1})\big(\|f\|_{L^\infty(\mathbb R^2)}+\|f\|_{L^1(\mathbb R^2)}\big)
\qquad (|x|<e).
\label{loc-bdd}
\end{equation}
We thus obtain $u\in L^\infty_{loc}(\mathbb R^2)\cap {\cal S}^\prime(\mathbb R^2)$.

We next divide \eqref{vol} into three parts:
\begin{equation*}
\begin{split}
u(x) 
&=U_1(x)+U_2(x)+U_3(x) \\ 
&:=\left(\int_{|y|<|x|/2}+\int_{|x|/2\leq |y|\leq 2|x|}+\int_{|y|>2|x|}\right) 
\Gamma_a(x,y)f(y)\,dy. 
\end{split} 
\end{equation*}
By \eqref{asy-xy} and \eqref{f-cond-1} we have
\begin{equation} 
U_1(x) 
=\frac{x^\perp}{4\pi|x|^2}
\int_{|y|<|x|/2}y^\perp\cdot f(y)\,dy+W(x)
\label{asy-u1}
\end{equation}
with
\begin{equation}
\begin{split}
|W(x)| 
&\leq C|a|^{-1}|x|^{-2}\int_{|y|<|x|/2}|f(y)|\,dy
+C|x|^{-2}\int_{|y|<|x|/2}|y|^2|f(y)|\,dy \\ 
&\leq C|a|^{-1}|x|^{-2}\|f\|_{L^1(\mathbb R^2)}
+C|x|^{-1}\int_{\mathbb R^2}|y||f(y)|\,dy.
\end{split}
\label{remain-1}
\end{equation}
The second term of the first line of \eqref{remain-1} can be estimated even by
\[
C|x|^{-2}\int_0^{|x|/2}\big(\log\,(e+r)\big)^{-1}\,dr=o(|x|^{-1}) \quad
\mbox{as $|x|\to\infty$}.
\]
Note that this holds true under weaker assumption 
$f(x)=o(|x|^{-3})$ than $\mbox{\eqref{f-cond-1}}_2$.
This together with
\[
\left|\frac{x^\perp}{4\pi |x|^2}\int_{|y|\geq |x|/2}y^\perp\cdot f(y)\,dy\right|
\leq \frac{C}{|x|}\int_{|y|\geq |x|/2}|y||f(y)|\,dy
=o(|x|^{-1})
\]
implies that
\begin{equation}
U_1(x)
=\frac{x^\perp}{4\pi |x|^2}\int_{\mathbb R^2}y^\perp\cdot f(y)\,dy
+o(|x|^{-1}) \quad\mbox{as $|x|\to\infty$}.
\label{u-asy-1}
\end{equation}
Let $|x|\geq e$, then
it follows from \eqref{loc-sum-y} with $\rho=2|x|$ and \eqref{f-cond-1} that
\begin{equation}
\begin{split}
|U_2(x)| 
&\leq\int_{|x|/2\leq |y|\leq 2|x|}
|\Gamma_a(x,y)||f(y)|\,dy  \\
&\leq C|x|^{-3}\big(\log\,\frac{|x|}{2}\big)^{-1}
\int_{|y|\leq 2|x|}|\Gamma_a(x,y)|\,dy 
\sup_{|y|\geq |x|/2}|y|^3(\log |y|)|f(y)|  \\
&\leq C(1+|a|^{-1})|x|^{-1}
\sup_{|y|\geq |x|/2}|y|^3(\log |y|)|f(y)| \qquad
(|x|\geq e).
\end{split}
\label{remain-2}
\end{equation}
Under stronger assumption \eqref{f-cond-2}, we see that
$U_2(x)=o(|x|^{-1})$ as $|x|\to\infty$.
We remark that \eqref{f-cond-2} is needed only here.
We use \eqref{asy-xy2} to find
\begin{equation}
\begin{split}
|U_3(x)| 
&\leq C\int_{|y|>2|x|}\left(\frac{|x|}{|y|}+\frac{1}{|a||y|^2}\right)|f(y)|\,dy  \\
&\leq C(|x|^{-1}+|a|^{-1}|x|^{-3})\int_{|y|>2|x|}|y||f(y)|\,dy
=o(|x|^{-1})
\end{split}
\label{remain-3}
\end{equation}
as $|x|\to\infty$.
We gather \eqref{asy-u1}, \eqref{remain-1}, \eqref{remain-2} and \eqref{remain-3} 
to conclude \eqref{est-u-far} for every $\rho \geq e$.
Then \eqref{est-u-far} with $\rho =e$
together with \eqref{loc-bdd} for $|x|<e$ yields 
\begin{equation}
\begin{split}
&\quad \sup_{x\in\mathbb R^2}(1+|x|)|u(x)|  \\
&\leq C(1+|a|^{-1})\left[\int_{\mathbb R^2}(1+|x|)|f(x)|\,dx
+\sup_{x\in\mathbb R^2}\,(1+|x|^3)\big(\log\,(e+|x|)\big)|f(x)|\right].
\end{split}
\label{est-u}
\end{equation}
Furthermore, we collect \eqref{u-asy-1}, \eqref{remain-2}
and \eqref{remain-3} to find the asymptotic representation
\eqref{asy-u} as long as \eqref{f-cond-2} is additionally imposed.
If in particular $\mbox{Supp $f$}\subset B_\rho(0)$ 
for some $\rho>0$, then $u(x)=U_1(x)$ for $|x|\geq 2\rho$.
In view of the first line of \eqref{remain-1},
we have
\[
|W(x)|\leq C(|a|^{-1}+\rho^2)|x|^{-2}\int_{|y|<\rho}|f(y)|\,dy=O(|x|^{-2}) \quad
\mbox{as $|x|\to\infty$}.
\]

To show the decay of $\nabla u(x)$, consider
\begin{equation*}
\begin{split}
V(x)&:=\int_{\mathbb R^2}\nabla_x\Gamma_a(x,y)f(y)\,dy  \\
&=\int_{|y|<|x|/2}+\int_{|x|/2\leq |y|\leq 2|x|}
+\int_{|y|>2|x|}  
=:V_1(x)+V_2(x)+V_3(x).
\end{split}
\end{equation*}
Neglecting the oscillation and using \eqref{grad-funda}--\eqref{K-grad} 
together with \eqref{exa-1}--\eqref{exa-2}, we deduce
\begin{equation}
|\nabla_x\Gamma_a(x,y)|\leq
\left\{
\begin{array}{ll}
C|x|^{-1}, \qquad &|x|>2|y|, \\
C|y|^{-1}, &|y|>2|x|.
\end{array}
\right.
\label{grad}
\end{equation}
Although they are not sharp (Remark \ref{rem-grad}), they respectively yield
\[
|V_1(x)|\leq C|x|^{-1}\|f\|_{L^1(\mathbb R^2)}
\]
and
\[
|V_3(x)|\leq C\int_{|y|>2|x|}|y|^{-1}|f(y)|\,dy
\leq C|x|^{-2}\int_{|y|>2|x|}|y||f(y)|\,dy.
\]
Let $|x|\geq e$ and use \eqref{loc-sum-grad} with $\rho=2|x|$ to find
\[
|V_2(x)|\leq C|x|^{-3}\big(\log\,\frac{|x|}{2}\big)^{-1}
\int_{|y|\leq 2|x|}|\nabla_x\Gamma(x,y)|\,dy
\leq C|x|^{-2}\big(\log\,\frac{|x|}{2}\big)^{-1}.
\]
We thus obtain
\[
|V(x)|\leq \frac{C}{|x|} \qquad (|x|\geq e).
\]
In order to conclude
$\nabla u(x)=O(|x|^{-1})$ as $|x|\to\infty$,
it suffices to show that
\begin{equation}
\nabla u=V \qquad
\mbox{in ${\cal D}^\prime(\mathbb R^2\setminus\overline{B_e(0)})$}.
\label{grad-vol}
\end{equation}
Given $\varphi\in C_0^\infty(\mathbb R^2\setminus\overline{B_e(0)})$
arbitrarily, we have
\[
\langle u, \mbox{div $\varphi$}\rangle
=\left\langle\int_{\mathbb R^2}\Gamma_a(\cdot,y)f(y)\,dy, \mbox{div $\varphi$}
\right\rangle
=\int_{\mathbb R^2}\langle\Gamma_a(\cdot,y), \mbox{div $\varphi$}
\rangle\,f(y)\,dy,
\]
in which the last equality is correct because
\[
\int_{e<|x|<M}\int_{\mathbb R^2}|\Gamma_a(x,y)||f(y)|
|\mbox{div $\varphi(x)$}|\,dy\,dx
\leq C\int_{e<|x|<M}\frac{|\mbox{div $\varphi(x)$}|}{|x|}\,dx <\infty
\]
follows from the proof of \eqref{est-u-far},
where $\mbox{Supp $\varphi$}\subset B_M(0)\setminus\overline{B_e(0)}$.
We further obtain
\[
\langle u, \mbox{div $\varphi$}\rangle
=-\int_{\mathbb R^2}\langle\nabla_x\Gamma_a(\cdot,y), \varphi\rangle\,f(y)\,dy
=-\langle V, \varphi\rangle
\]
since we have
\[
\int_{e<|x|<M}\int_{\mathbb R^2}|\nabla_x\Gamma_a(x,y)||f(y)| 
|\varphi(x)|\,dy\,dx
\leq C\int_{e<|x|<M}\frac{|\varphi(x)|}{|x|}\,dx <\infty
\]
by computation as above.
We are thus led to \eqref{grad-vol}.

We turn to the decay property of the pressure
\[
p(x)=\frac{x}{2\pi|x|^2}\cdot\int_{\mathbb R^2}f(y)\,dy +R(x),
\]
where the remainder $R(x)$ is divided into three parts:
\begin{equation*}
\begin{split}
R(x) 
&=R_1(x)+R_2(x)+R_3(x) \\   
&:=\frac{1}{2\pi} 
\left(\int_{|y|<|x|/2}+\int_{|x|/2\leq |y|\leq 2|x|}+\int_{|y|>2|x|}\right) 
\left(\frac{x-y}{|x-y|^2}-\frac{x}{|x|^2}\right)\cdot f(y)\,dy.
\end{split} 
\end{equation*}
We then observe
\[
|R_1(x)|  
\leq\frac{1}{2\pi}\int_{|y|<|x|/2}\int_0^1\frac{3|y|}{|x-ty|^2}\,dt\,|f(y)|\,dy
\leq C|x|^{-2}\int_{\mathbb R^2}|y||f(y)|\,dy
\]
and
\begin{equation*}
\begin{split} 
|R_2(x)|
&\leq C|x|^{-3}\big(\log\,\frac{|x|}{2}\big)^{-1}
\left(\int_{|y-x|\leq 3|x|}\frac{1}{|x-y|}\,dy
+\frac{1}{|x|}\int_{|y|\leq 2|x|}dy\right) \\
&=C|x|^{-2}\big(\log\,\frac{|x|}{2}\big)^{-1}
\end{split} 
\end{equation*}
as well as
\begin{equation*} 
\begin{split} 
|R_3(x)| 
&\leq\frac{1}{2\pi}\int_{|y|>2|x|}
\left(\frac{1}{|x-y|}+\frac{1}{|x|}\right)\,|f(y)|\,dy  \\
&\leq C|x|^{-2}\int_{|y|>2|x|}|y||f(y)|\,dy=o(|x|^{-2})
\end{split}
\end{equation*}
as $|x|\to\infty$.
We thus obtain \eqref{asy-p}.

We will show that \eqref{vol} is a solution to \eqref{rot-st-wh}.
We use $v^0$ and $v^1$ given by
\eqref{evolu-0} and \eqref{evolu-1},
which satisfy \eqref{cp-0} and \eqref{cp-1}, respectively, by \eqref{f-basic}.
We set
\[
v(x,t)=v^0(x,t)+v^1(x,t), \qquad
w(x)=\int_0^\infty v(x,t)\,dt.
\]
Since neither $u$ nor $w$ can absolutely converge,
we are unable to apply the Fubini theorem directly to them.
We will show, nevertheless, that they do converge and coincide.
Let us employ the centering technique as in \eqref{tilde}.
We consider
\[
\widetilde u(x)=\int_{\mathbb R^2}\widetilde \Gamma_a(x,y)f(y)\,dy,
\]
and
\[
\widetilde v(x,t)
=O(at)^T\int_{\mathbb R^2}
\left(K(O(at)x-y,t)-\frac{e^{-1/4t}}{8\pi t}\mathbb I\right)f(y)\,dy,
\]
\[
\widetilde w(x)=\int_0^\infty \widetilde v(x,t)\,dt,
\]
where $\widetilde \Gamma_a(x,y)$ is given by \eqref{tilde}.
Then both integrals of $\widetilde u$ and $\widetilde w$
are absolutely convergent over $(0,\infty)\times \mathbb R^2$
with respect to $(t,y)$.
In fact, as in \eqref{trans}, it follows from \eqref{cen-sp}--\eqref{cen-2}
together with the assumption \eqref{f-cond-1} that
\begin{equation*}
\begin{split}
&\quad \int_0^\infty\int_{\mathbb R^2}
\left|K(O(at)x-y,t)-\frac{e^{-1/4t}}{8\pi t}\mathbb I\right||f(y)|\,dy\,dt \\
&=\int_0^\infty\int_{\mathbb R^2} 
\left|K(x-y,t)-\frac{e^{-1/4t}}{8\pi t}\mathbb I\right||f(O(at)y)|\,dy\,dt \\
&\leq C\int_{\mathbb R^2}
\frac{\left(\big|\log |x-y|\big|+1\right)}{1+|y|^3}\,dy
\end{split}
\end{equation*}
which is actually convergent.
We thus obtain $\widetilde u=\widetilde w$.
Since
\begin{equation}
u-\widetilde u
=\int_0^\infty O(at)^T e^{-1/4t}\frac{dt}{8\pi t}\int_{\mathbb R^2}f(y)\, dy
=w-\widetilde w
\label{zure}
\end{equation}
and since \eqref{zure} does converge by \eqref{subt},
we eventually conclude that $u=w$.
We now show that $\{u,p\}$ actually satisfies $\mbox{\eqref{rot-st-wh}}_1$
in the sense of distributions.
Given $\varphi\in C_0^\infty(\mathbb R^2)$ arbitrarily,
let us consider 
$\langle \widetilde u, L_{-a}\varphi\rangle$
since we have the adjoint relation $L_{-a}=L_a^*$, see \eqref{diff-op}.
Then we find
\[
\langle\widetilde u, L_{-a}\varphi\rangle
=\langle\widetilde w, L_{-a}\varphi\rangle
=\int_0^\infty\langle \widetilde v(t), L_{-a}\varphi\rangle\,dt,
\]
in which the Fubini theorem is employed.
Note that the argument does not work
if $\widetilde v$ is replaced by $v$.
By integration by parts we have
\begin{equation}
\langle\widetilde u, L_{-a}\varphi\rangle
=\int_0^\infty\langle L_a v(t), \varphi\rangle\,dt
+\int_0^\infty
\left\langle L_a\big(\widetilde v(t)-v(t)\big), \varphi\right\rangle\,dt,
\label{zure-2}
\end{equation}
however, since $\widetilde v-v$ is independent of $x$ and since
\[
\int_{\mathbb R^2}(L_{-a}\varphi)(x)\,dx
=-a\int_{\mathbb R^2}\varphi^\perp(x)\,dx,
\]
we obtain
\begin{equation}
\begin{split}
\int_0^\infty\left\langle L_a\big(\widetilde v(t)-v(t)\big), \varphi\right\rangle\,dt
&=-a(u-\widetilde u)^\perp\cdot\int_{\mathbb R^2}\varphi(x)\,dx \\
&=a(u-\widetilde u)\cdot\int_{\mathbb R^2}\varphi^\perp(x)\,dx \\
&=-\langle u-\widetilde u, L_{-a}\varphi\rangle,
\end{split}
\label{zure-3}
\end{equation}
see \eqref{zure}.
On the other hand, in view of \eqref{cp-0} and \eqref{cp-1} and by taking
\[
\lim_{t\to\infty}\;\langle v(t), \varphi\rangle =0, \qquad
\lim_{t\to 0}\;\langle v(t)-(f-\nabla p), \varphi\rangle =0,
\]
into account, we have
\begin{equation}
\int_0^\infty\langle L_a v(t), \varphi\rangle\,dt
=-\int_0^\infty\partial_t \langle v(t), \varphi\rangle\,dt
=\langle f-\nabla p, \varphi\rangle.
\label{from-IC}
\end{equation}
We collect \eqref{zure-2}, \eqref{zure-3} and \eqref{from-IC} to obtain
\[
\langle u, L_{-a}\varphi\rangle=\langle f-\nabla p, \varphi\rangle
\]
for all $\varphi\in C_0^\infty(\mathbb R^2)$.
Since $\Delta p=\mbox{div $f$}$,
we take the divergence of $\mbox{\eqref{rot-st-wh}}_1$ to see that
$(\mbox{div $u$})\in {\cal S}^\prime(\mathbb R^2)$ obeys
\[
-\Delta (\mbox{div $u$})-a\,x^\perp\cdot\nabla (\mbox{div $u$})=0
\]
on account of \eqref{solenoidal}.
By Lemma \ref{f-supp}, $\mbox{div $u$}$ is a polynomial,
however, from \eqref{decay-rate} we conclude that $\mbox{div $u$}=0$.
Since $f\in L^q(\mathbb R^2)$ for every $q\in (1,\infty)$,
the result of \cite{FHM}
(see also another proof given by \cite{GK})
implies \eqref{sum-property1}.
And then,
\eqref{est-u} combined with \eqref{sum-property1} especially
for $x^\perp\cdot\nabla u-u^\perp$
leads to \eqref{sum-property2}.
Finally, if $f\in C^\infty(\mathbb R^2)$,
then we put the term $x^\perp\cdot\nabla u-u^\perp$ in the RHS together with
such $f$ to use the regularity theory
of the usual Stokes system.
As a consequence, we find $\{u,p\}\in C^\infty(\mathbb R^2)$.
This completes the proof.
\hfill
$\Box$
\bigskip

For the proof of Theorem \ref{existence} we also need analysis of
the system
\begin{equation}
\varepsilon u-\Delta u-a\big(x^\perp\cdot\nabla u-u^\perp\big)+\nabla p=f, \qquad
\mbox{div $u$}=0 \qquad\mbox{in $\mathbb R^2$},
\label{rot-st-ep}
\end{equation}
where the term $\varepsilon u$ is introduced
in order to control the behavior of solutions at infinity.
Indeed \eqref{rot-st-ep} is the resolvent system,
but the only case we are going to consider is $\varepsilon >0$.
The velocity part of the associated fundamental solution is given by
\begin{equation}
\Gamma^{(\varepsilon)}_a(x,y)
=\int_0^\infty e^{-\varepsilon t}O(at)^TK(O(at)x-y,t)\,dt,
\label{ep-funda}
\end{equation}
while the pressure part is the same, see \eqref{pre-funda}.
Of course, \eqref{ep-funda} converges without using oscillation, however,
what we need is to derive a certain estimate uniformly with respect to 
$\varepsilon >0$.
Therefore, we still use oscillation as well as the centering technique.
\begin{proposition}
Let $a\in \mathbb R\setminus\{0\}$.
Suppose $f$ satisfies \eqref{f-basic} and \eqref{f-cond-1}.
Set
\begin{equation}
u_\varepsilon(x)=\int_{\mathbb R^2}
\Gamma^{(\varepsilon)}_a(x,y)f(y)\,dy, \qquad
\varepsilon >0,
\label{ep-vol}
\end{equation}
where $\Gamma^{(\varepsilon)}_a(x,y)$ is given by \eqref{ep-funda}.
Then $u_\varepsilon(x)$ enjoys \eqref{est-u-far}
for every $\rho\geq e$,
where the constant $C>0$ is independent of $\varepsilon >0$
(as well as $\rho\in [e,\infty)$ and $a\in \mathbb R\setminus\{0\}$).
Furthermore, the pair $\{u_\varepsilon,p\}$
is a solution to \eqref{rot-st-ep} in the sense of distributions,
where $p(x)$ given by \eqref{vol-p}.
\label{ep-decay}
\end{proposition}

\begin{proof}
Let $m>0$.
As in the proof of \eqref{osc-1}--\eqref{osc-2}
by use of integration by parts,
we easily find
\begin{equation}
\begin{split}
&\quad \left|\int_0^\infty e^{-\varepsilon t+iat}e^{-r^2/t}\,\frac{dt}{t^m}\right|
+\left|\int_0^\infty e^{-\varepsilon t+iat}
\int_t^\infty e^{-r^2/s}\,\frac{ds}{s^{m+1}}\,dt\right|  \\
&\leq\frac{C}{\sqrt{\varepsilon^2+a^2}\; r^{2m}}
\leq\frac{C}{|a|r^{2m}}
\end{split}
\label{ep-osc}
\end{equation}
with some $C=C(m)>0$ independent of
$\varepsilon \geq 0$, $r>0$ and $a\in\mathbb R\setminus\{0\}$.
Owing to \eqref{ep-osc}, we have the similar estimates to 
\eqref{loc-sum-y}, \eqref{asy-xy} and \eqref{asy-xy2}
uniformly in $\varepsilon >0$; namely, there is a constant $C>0$
independent of $\varepsilon >0$ such that
\begin{equation}
\begin{split}
&\int_{|y|\leq 2|x|}|\Gamma^{(\varepsilon)}_a(x,y)|\,dy
\leq C|a|^{-1}|x|^2+C|x|^2\log |x|, \qquad |x|\geq e, \\
&|\Gamma^{(\varepsilon)}_a(x,y)|
\leq
\left\{
\begin{array}{ll}
C|x|^{-1}|y|+C|a|^{-1}|x|^{-2}, \qquad &|x|>2|y|, \\
C|y|^{-1}|x|+C|a|^{-1}|y|^{-2}, &|y|>2|x|.
\end{array}
\right.
\end{split}
\label{ep-est}
\end{equation}
In fact, it follows from \eqref{ep-osc} that
\[
\left|\int_0^\infty e^{-\varepsilon t}O(at)^T\,\frac{e^{-1/4t}}{8\pi t}\,dt\right|
\leq\frac{C}{|a|}
\]
with some $C>0$ independent of $\varepsilon >0$,
which together with the same computing as in the proof of Lemma \ref{loc-sum}
by means of centering technique as in \eqref{tilde}
yields $\mbox{\eqref{ep-est}}_1$.
Also, look at the proof of Proposition \ref{key},
in which oscillation is used in \eqref{decay-01} and so on.
This time, we employ \eqref{ep-osc} to get 
\[
\left|\int_0^\infty e^{-\varepsilon t}O(at)^T\,e^{-|x|^2/4t}\frac{dt}{t}\right|
\leq\frac{C}{|a||x|^2}
\]
and so on, where $C>0$ is independent of $\varepsilon >0$.
The other estimates without using oscillation are obvious.
For the purpose here it is enough to split
the exponential function into two terms rather than \eqref{taylor} and
\eqref{taylor2} since we do not intend to find out the leading term.
As a consequence, we obtain $\mbox{\eqref{ep-est}}_2$.
With use of \eqref{ep-est}
the desired estimate \eqref{est-u-far} for $u_\varepsilon$ uniformly in 
$\varepsilon >0$ is deduced in exactly the same way as in 
the proof of Proposition \ref{justify}.

The proof of the latter assertion is easier than the corresponding part
(the 4th assertion) of Proposition \ref{justify},
in which we are forced to introduce $\widetilde u$.
We do not need it since $u_\varepsilon$ itself converges absolutely.
Hence, we have
\[
u_\varepsilon(x)=\int_0^\infty v_\varepsilon(x,t)\,dt
\]
with
\[
v_\varepsilon(x,t)
=e^{-\varepsilon t}O(at)^T
\int_{\mathbb R^2}K(O(at)x-y,t)f(y)\,dy,
\]
which satisfies
\[
\partial_tv_\varepsilon+(\varepsilon +L_a)v_\varepsilon=0, \qquad
v_\varepsilon(\cdot,0)=f-\nabla p
\]
in $\mathbb R^2\times (0,\infty)$, where $L_a$ is given by \eqref{diff-op}.
We thus obtain
\[
\langle u_\varepsilon, (\varepsilon +L_{-a})\varphi\rangle
=\langle f-\nabla p, \varphi\rangle
\]
for all $\varphi\in C_0^\infty(\mathbb R^2)$.
This combined with $\Delta p=\mbox{div $f$}$ implies
$\mbox{div $u_\varepsilon$}=0$ by Lemma \ref{f-supp} since
$|\nabla u_\varepsilon(x)|=O(|x|^{-1})$ as $|x|\to\infty$,
where this decay property is verified along the same line as the case 
$\varepsilon =0$ by use of \eqref{loc-sum-grad} and \eqref{grad} 
for $\nabla_x\Gamma^{(\varepsilon)}_a(x,y)$ without using oscillation.
The proof is complete.
\end{proof}

\section{Proof of Theorem \ref{decay}}
\label{proof-1}

To find the asymptotic representation \eqref{asym-rep},
it would be standard to employ a potential representation formula
in terms of the fundamental solution 
$\Gamma_a(x,y)$ as in \cite{FH1} for the 3D problem,
but we have to establish the decay properties \eqref{resolution}
in advance in order to justify such a formula.
This procedure consisting of those two steps 
would be also fine (and actually it works),
however, there is another way, 
which is straightforward and leads us directly
to \eqref{asym-rep} as well as \eqref{resolution},
by means of a cut-off technique.
We will adopt the latter way to prove Theorem \ref{decay}.
The only disadvantage compared with the former one by use of
the potential representation formula
is that the coefficient of the leading profile needs
a bit lengthy (but elementary) calculation.
\bigskip
 
\noindent
{\it Proof of Theorem \ref{decay}}.
We use a cut-off technique as mentioned above.
In order to recover the solenoidal condition by use of
the correction term with compact support, 
we first reduce the problem to the one with vanishing flux
at the boundary $\partial\Omega$.
To this end, we fix
$x_0\in\mbox{int $(\mathbb R^2\setminus\Omega)$}$ and introduce
the flux carrier
\[
w(x)=\beta\,\nabla\left(\frac{1}{2\pi}\log\frac{1}{|x-x_0|}\right)
=\frac{-\beta\,(x-x_0)}{2\pi |x-x_0|^2}, \qquad
\beta=\int_{\partial\Omega}\nu\cdot u\,d\sigma,
\]
for given smooth solution $\{u,p\}$ of \eqref{rot-stokes}.
Then we have
\[
\int_{\partial\Omega}\nu\cdot w\,d\sigma=\beta,
\]
\begin{equation}
\mbox{div $w$}=0, \qquad
\Delta w=0, \qquad
(x-x_0)^\perp\cdot\nabla w=w^\perp \qquad\mbox{in $\mathbb R^2\setminus\{x_0\}$}
\label{carr}
\end{equation}
and
\begin{equation}
\nabla^j w(x)=\nabla^j\left(\frac{-\beta x}{2\pi |x|^2}\right)
+O(|x|^{-(2+j)}) \qquad (j=0, 1)
\label{carrier}
\end{equation}
as $|x|\to\infty$.
So the pair
\[
\widetilde u=u-w, \qquad
\widetilde p=p-a\,x_0^\perp\cdot w
\]
fulfills \eqref{rot-stokes} subject to
\begin{equation}
\int_{\partial\Omega}\nu\cdot\widetilde u\,d\sigma=0,
\label{zero-fl}
\end{equation}
where we note the relation
\begin{equation}
\partial_k (x_0^\perp\cdot w)
=\sum_j(x_0^\perp)_j 
\partial_k\partial_j\left(\frac{\beta}{2\pi}\log\frac{1}{|x-x_0|}\right)
=x_0^\perp\cdot\nabla w_k  \qquad (k=1,2).
\label{car-grad}
\end{equation}

We fix $R\geq 1$ such that
$\mathbb R^2\setminus\Omega\subset B_R(0)$.
Let 
$\psi\in C_0^\infty(B_{3R}(0); [0,1])$
be a cut-off function satisfying
$\psi(x)=1$ for $|x|\leq 2R$.
By using the Bogovskii operator $B$ in the annulus
\[
A=\{x\in\mathbb R^2;\, R<|x|<3R\},
\]
see \cite{B}, \cite{BS} and \cite{Ga-b}, we set
\[
v=(1-\psi)\widetilde u +B[\widetilde u\cdot\nabla\psi],\qquad
q=(1-\psi)\widetilde p.
\]
It should be noted that
$\int_A \widetilde u\cdot\nabla\psi\,dx=0$ follows from \eqref{zero-fl}.
Then the pair $\{v,q\}$ obeys
\begin{equation}
-\Delta v-a\left(x^\perp\cdot\nabla v-v^\perp\right)+\nabla q=g+(1-\psi)f, \qquad
\mbox{div $v$}=0 \qquad\mbox{in $\mathbb R^2$}
\label{whole} 
\end{equation}
for some function $g\in C_0^\infty(\mathbb R^2)$
whose support is a compact set in $A$.
Here, we do not need any explicit form of $g$;
in fact, the important quantity \eqref{alpha-0} below can be calculated only
by taking account of the structure of the equation \eqref{whole},
that is,
$\mbox{div $S(v,q)$}=-g-(1-\psi)f$, see \eqref{new-stress}.
When
$u(x)=o(|x|)$, it is obvious that $v\in {\cal S}^\prime(\mathbb R^2)$.
Under the alternative assumption
$\nabla u\in L^r(\Omega\setminus B_R(0))$ for some $r<\infty$,
we have 
$\nabla v\in {\cal S}^\prime(\mathbb R^2)$,
which implies
$v\in {\cal S}^\prime(\mathbb R^2)$
by \cite[Proposition 1.2.1]{C}.
Going back to \eqref{whole}, we observe
$\nabla q\in {\cal S}^\prime(\mathbb R^2)$ and thereby
$q\in {\cal S}^\prime(\mathbb R^2)$, too.
Proposition \ref{justify} together with Lemma \ref{f-supp}
concludes that
\begin{equation}
\begin{split}
&v(x)=\int_{\mathbb R^2}\Gamma_a(x,y)\{g+(1-\psi)f\}(y)\,dy +P_v(x),  \\
&q(x)=\int_{\mathbb R^2}Q(x-y)\cdot \{g+(1-\psi)f\}(y)\,dy+P_q(x),
\end{split} 
\label{tempered}
\end{equation}
with some polynomials $P_v$ and $P_q$, however,
it turns out from \eqref{decay-rate} and
from either
$\nabla v\in L^r(\mathbb R^2)$ with some $r\in (1,\infty)$ or
$v(x)=o(|x|)$
that $P_v$ must be a constant vector $u_\infty$.
Thus we have
\begin{equation}
u(x)
=w(x)+\int_{\mathbb R^2}\Gamma_a(x,y)\{g+(1-\psi)f\}(y)\,dy +u_\infty \qquad 
(|x|\geq 3R),
\label{u-far}
\end{equation}
from which combined with
\eqref{asy-u} and \eqref{carrier} we obtain
\eqref{asym-rep} under the additional condition \eqref{f-cond-2}
as well as $\mbox{\eqref{resolution}}_1$,
where the coefficient
\begin{equation}
\alpha=\int_{\mathbb R^2}y^\perp\cdot \{g+(1-\psi)f\}(y)\,dy
=-\int_{\mathbb R^2} y^\perp\cdot\mbox{div $S(v,q)$}\,dy
\label{alpha-0}
\end{equation}
is computed as follows.

Set
\[
\alpha(\rho)
:=-\int_{|y|<\rho} y^\perp\cdot\mbox{div $S(v,q)$}\,dy  \qquad (\rho >3R).
\]
In view of \eqref{new-stress} we have the relation
\begin{equation}
\begin{split}
\mbox{div $\big(y^\perp\cdot S(v,q)\big)$}
&=\sum_{j,k}\partial_k\left[(y^\perp)_jS_{jk}(v,q)\right] \\
&=y^\perp\cdot\mbox{div $S(v,q)$}-S_{12}(v,q)+S_{21}(v,q) \\
&=y^\perp\cdot\mbox{div $S(v,q)$}-2a\, y\cdot v
\end{split}
\label{div-1}
\end{equation}
to find
\[
\alpha(\rho)=-\int_{|y|=\rho}
y^\perp\cdot\left(S(\widetilde u,\widetilde p)\,\frac{y}{\rho}\right)\,d\sigma
-2a\int_{|y|<\rho} y\cdot v\,dy.
\]
Since
$\mbox{div $S(\widetilde u,\widetilde p)$}=-f$ in $\Omega$,
it follows from \eqref{div-1}
in which $v$ is replaced by $\widetilde u$ that
\[
\alpha(\rho)=\int_{\partial\Omega}
y^\perp\cdot\left(S(\widetilde u,\widetilde p)\nu\right)\,d\sigma
+2a\int_{\Omega_\rho}y\cdot (\widetilde u-v)\,dy
+\int_{\Omega_\rho}y^\perp\cdot f\,dy.
\]
We are going to compute
\begin{equation*}
\begin{split}
&\quad \int_{\partial\Omega}
y^\perp\cdot\left(S(\widetilde u,\widetilde p)\nu\right)\,d\sigma  \\
&=\int_{\partial\Omega}
y^\perp\cdot\left\{\left(T(u,p)+a\,u\otimes y^\perp\right)\nu\right\}\,d\sigma  \\
&\quad -\int_{\partial\Omega}y^\perp\cdot\left((Dw)\nu\right)\,d\sigma
+a\int_{\partial\Omega}(y^\perp\cdot\nu)(x_0^\perp\cdot w)\,d\sigma  \\
&\quad -a\int_{\partial\Omega}y^\perp\cdot
\left\{(w\otimes y^\perp)\nu\right\}\,d\sigma
-a\int_{\partial\Omega}y^\perp\cdot
\left\{(y^\perp\otimes \widetilde u)\nu\right\}\,d\sigma  \\
&=:\int_{\partial\Omega}
y^\perp\cdot\left\{\left(T(u,p)+a\,u\otimes y^\perp\right)\nu\right\}\,d\sigma
+J_1+J_2+J_3+J_4.
\end{split}
\end{equation*}
We will show that
\[
J_1=0, \qquad
J_2+J_3=0, \qquad
J_4+2a\int_{\Omega_\rho}y\cdot (\widetilde u-v)\,dy=0,
\]
which concludes
\[
\alpha(\rho)
=\int_{\partial\Omega}
y^\perp\cdot\left\{\left(T(u,p)+a\,u\otimes y^\perp\right)\nu\right\}\,d\sigma 
+\int_{\Omega_\rho} y^\perp\cdot f\,dy.
\]
Letting $\rho\to\infty$ leads us to
\begin{equation}
\alpha
=\int_{\partial\Omega} 
y^\perp\cdot\left\{\left(T(u,p)+a\,u\otimes y^\perp\right)\nu\right\}\,d\sigma
+\int_\Omega y^\perp\cdot f\,dy.
\label{alpha-1}
\end{equation}
In fact, we observe
\begin{equation*}
\begin{split}
2a\int_{\Omega_\rho}y\cdot(\widetilde u-v)\,dy
&=a\int_{\Omega_\rho}
\left\{\psi\widetilde u-B[\widetilde u\cdot\nabla\psi]\right\}
\cdot\nabla |y|^2\,dy   \\
&=a\int_{\Omega_\rho}
\mbox{div}\,\big[|y|^2\left\{\psi\widetilde u-B[\widetilde u\cdot\nabla\psi]\right\}
\big]\,dy  \\
&=a\int_{\partial\Omega}|y|^2(\nu\cdot\widetilde u)\,d\sigma=-J_4
\end{split}
\end{equation*}
and
\[
J_2+J_3
=-a\int_{\partial\Omega}(y^\perp\cdot \nu)(y-x_0)^\perp\cdot w\,d\sigma=0
\]
on account of 
$(y-x_0)^\perp\cdot w(y)=0$.
We take account of $(\nabla w)^T=\nabla w$ and $\Delta w=0$ in 
$\mathbb R^2\setminus\{x_0\}$ to see that
\begin{equation*}
\begin{split}
J_1
&=-2\int_{\partial\Omega}(y-x_0)^\perp\cdot(\nu\cdot\nabla w)\,d\sigma
-2x_0^\perp\cdot\int_{\partial\Omega}\nu\cdot\nabla w\,d\sigma   \\
&=-2\int_{\partial\Omega}(y-x_0)^\perp\cdot(\nu\cdot\nabla w)\,d\sigma
+2x_0^\perp\cdot\int_{|y-x_0|=\varepsilon}
\frac{y-x_0}{\varepsilon}\cdot\nabla w\,d\sigma
\end{split}
\end{equation*}
where $\varepsilon >0$ is taken in such a way that
$\overline{B_\varepsilon(x_0)} \subset\mbox{int $(\mathbb R^2\setminus\Omega)$}$.
Using the explicit representation
\[
\nabla w(y)=\frac{-\beta}{2\pi}
\left(\frac{\mathbb I}{|y-x_0|^2}-\frac{2(y-x_0)\otimes (y-x_0)}{|y-x_0|^4}\right),
\]
we find
\begin{equation*}
\begin{split}
\int_{\partial\Omega}(y-x_0)^\perp\cdot(\nu\cdot\nabla w)\,d\sigma
&=\frac{-\beta}{2\pi}\int_{\partial\Omega}
\frac{(y-x_0)^\perp\cdot\nu}{|y-x_0|^2}\,d\sigma  \\
&=\frac{\beta}{2\pi}
\int_{\mathbb R^2\setminus \big( \Omega\cup B_\varepsilon(x_0) \big)}
\mbox{div}\,\frac{(y-x_0)^\perp}{|y-x_0|^2}\,dy
=0
\end{split}
\end{equation*}
and
\[
\int_{|y-x_0|=\varepsilon}\frac{y-x_0}{\varepsilon}\cdot\nabla w\,d\sigma
=\frac{\beta}{2\pi\varepsilon^3}\int_{|y-x_0|=\varepsilon}(y-x_0)\,d\sigma
=0
\]
which implies that $J_1=0$.
We thus obtain \eqref{alpha-1}.

Concerning the pressure, 
it follows from \eqref{car-grad} and \eqref{tempered} that
\[
\nabla p
=a\,x_0^\perp\cdot\nabla w
+\nabla\int_{\mathbb R^2}Q(x-y)\cdot\{g+(1-\psi)f\}(y)\,dy
+\nabla P_q \quad (|x|\geq 3R).
\]
By \eqref{sum-property1} together with \eqref{carrier} we know
$\nabla (p-P_q)\in L^r(\Omega)$ for every $r\in (1,\infty)$.
Since $\Delta w=0$ in $\mathbb R^2\setminus\{x_0\}$, we obtain from \eqref{u-far}
\[
\Delta u=\Delta \int_{\mathbb R^2}\Gamma_a(x,y)\{g+(1-\psi)f\}(y)\,dy \qquad 
(|x|\geq 3R),
\]
so that $\Delta u\in L^r(\Omega)$ for every $r\in (1,\infty)$
on account of \eqref{sum-property1}.
In addition, we also have
\[
x^\perp\cdot\nabla u
=x^\perp\cdot\nabla w
+x^\perp\cdot\nabla\int_{\mathbb R^2}\Gamma_a(x,y)
\{g+(1-\psi)f\}(y)\,dy \qquad
(|x|\geq 3R).
\]
It thus follows from \eqref{sum-property2} and \eqref{carrier} that
$x^\perp\cdot\nabla u\in L^r(\Omega)$ for every $r\in (2,\infty)$.
Taking those as well as 
$\mbox{\eqref{resolution}}_1$ into account,
we go back to \eqref{rot-stokes}
and let $|x|\to\infty$ to find that
$\nabla P_q=-au_\infty^\perp$.
This implies that
\[
p=a\,x_0^\perp\cdot w
+\int_{\mathbb R^2}Q(x-y)\cdot\{g+(1-\psi)f\}(y)\,dy
-a\,u_\infty^\perp\cdot x+p_\infty  \qquad (|x|\geq 3R)
\]
for some constant $p_\infty$.
By \eqref{decay-rate} together with \eqref{carrier} we obtain
$\mbox{\eqref{resolution}}_2$.
We also use \eqref{asy-p} and carry out a bit computation to obtain
\begin{equation}
\begin{split}
p(x)+au_\infty^\perp\cdot x-p_\infty
&=\Big[
\int_{\partial\Omega}\left\{
T(\widetilde u,\widetilde p)
+a\,\big(\widetilde u\otimes y^\perp-y^\perp\otimes\widetilde u\big)
\right\}\nu\,d\sigma_y  \\
&\qquad +\int_\Omega f\,dy-\beta a x_0^\perp
\Big]\cdot\frac{x}{2\pi |x|^2}
+O(|x|^{-2})
\end{split}
\label{p-repre}
\end{equation}
as $|x|\to\infty$.
We stop further computation of the coefficient, however,
we will recall \eqref{p-repre} in Theorem \ref{existence},
in which the coefficient is much simplified.

Once we have fine decay properties \eqref{resolution},
we are able to justify the energy relation \eqref{energy}.
We first verify \eqref{skew}
for smooth vector fields $u,\, v\in H^1_{loc}(\overline{\Omega})$
without assuming their decay properties at infinity.
For each $\rho >0$ large enough we have
\begin{equation}
\begin{split}
\int_{\Omega_\rho}
\left[(x^\perp\cdot\nabla u-u^\perp)\cdot v
+u\cdot(x^\perp\cdot\nabla v-v^\perp)\right]\,dx
&=\int_{\Omega_\rho} \mbox{div $[x^\perp (u\cdot v)]$}\,dx  \\
&=\int_{\partial\Omega}(\nu\cdot x^\perp)(u\cdot v)\,d\sigma
\end{split}
\label{skew-2}
\end{equation}
since $\int_{|x|=\rho}=0$.
Letting $\rho\to\infty$, we obtain \eqref{skew}.
Now, given smooth solution 
$\{u,p\}\in H^1_{loc}(\overline{\Omega})\times L^2_{loc}(\overline{\Omega})$,
we use the constants $\{u_\infty,p_\infty\}$ found above and set
\[
u_*(x)=u(x)-u_\infty, \qquad
p_*(x)=p(x)+au_\infty^\perp\cdot x-p_\infty,
\]
which satisfy
\[
-\Delta u_*-a\,(x^\perp\cdot\nabla u_*-u_*^\perp)+\nabla p_*=f, \qquad
\mbox{div $u_*$}=0 \qquad\mbox{in $\Omega$}.
\]
We multiply $u_*$, perform integration by parts over $\Omega_\rho$
and use \eqref{skew-2} to find the following two equalities,
in which the relation
$\mbox{div $T(u_*,p_*)$}=\Delta u_*-\nabla p_*$
is used for the latter:
\begin{equation*}
\begin{split}
\int_{\Omega_\rho}|\nabla u_*|^2\,dx
&=\int_{\partial\Omega_\rho}
\big((\nabla u_*-p_*\mathbb I)\nu\big)\cdot u_* \,d\sigma
+I, \\
\frac{1}{2}\int_{\Omega_\rho}|Du_*|^2\,dx
&=\int_{\partial\Omega_\rho}
\big(T(u_*,p_*)\nu\big)\cdot u_*\,d\sigma
+I,
\end{split}
\end{equation*}
where the common term $I$ is given by
\[
I=\frac{a}{2}\int_{\partial\Omega}(\nu\cdot x^\perp)|u_*|^2\,d\sigma
+\int_{\Omega_\rho}f\cdot u_*\,dx.
\]
Note that both
$T(u_*,p_*)\nu$ and $(\nabla u_*-p_*\mathbb I)\nu$
are understood as the normal trace being in
$H^{-1/2}(\partial\Omega_\rho)$.
In view of \eqref{u-far}, we see from \eqref{decay-rate} and \eqref{carrier} that
\[
\nabla u_*(x)=\nabla u(x)=O(|x|^{-1}) \quad
\mbox{as $|x|\to\infty$}.
\]
This together with \eqref{resolution} implies that
\[
\lim_{\rho\to\infty}\int_{|x|=\rho}
\left(\frac{\partial u_*}{\partial\nu}\cdot u_*-(\nu\cdot u_*)p_*\right)\,d\sigma=0,
\]
and that
\[
\lim_{\rho\to\infty}\int_{|x|=\rho}
\big(T(u_*,p_*)\nu\big)\cdot u_*\,d\sigma=0.
\]
On the other hand, we know that
$f\cdot u_*\in L^1(\Omega)$ by \eqref{f-cond} and \eqref{resolution} together with
$u\in H^1_{loc}(\overline{\Omega})\subset L^s_{loc}(\overline{\Omega})$
for all $s<\infty$.
We thus obtain not only $\nabla u\in L^2(\Omega)$ but \eqref{energy}.
This completes the proof.
\hfill
$\Box$

\section{Proof of Theorem \ref{existence}}
\label{proof-2}

{\it Proof of Theorem \ref{existence}}.
We begin with the proof of uniqueness.
Suppose $\{u,p\}$ is a solution in the sense of distributions to
\eqref{rot-stokes} with $f=0$ subject to $u=0$ on $\partial\Omega$
and $\{u,p\}\to \{0,0\}$ as $|x|\to\infty$
within the class 
$\nabla u\in L^2(\Omega)$, $\{u,p\}\in L^2_{loc}(\overline{\Omega})$.
We put the term $x^\perp\cdot\nabla u-u^\perp$ in the RHS
and use the interior regularity theory of the usual Stokes system
to show that
$u\in H^{k+1}_{loc}(\Omega)$, $p\in H^k_{loc}(\Omega)$
for every integer $k\geq 1$ by bootstrapping argument;
hence, $u,\, p\in C^\infty(\Omega)$.
By Theorem \ref{decay} we have \eqref{energy}, in which the RHS vanishes.
So, $u$ is the rigid motion, but $u=0$ on account of the boundary condition.
Going back to \eqref{rot-stokes} (with $f=0$), 
we have $\nabla p=0$, which together with 
$p\to 0$ at infinity yields $p=0$.
This proves the uniqueness.

We turn to the existence.
It is easy to find a solution with $\nabla u\in L^2(\Omega)$
by following the method of Leray,
but one cannot exclude a constant vector $u_\infty$ at infinity
even if applying Theorem \ref{decay}.
To get around this difficulty, we will adopt an approximation
procedure specified below which brings regularizing effect at infinity.
We take the auxiliary function
\begin{equation}
w(x)=\frac{a}{2}\,\nabla^\perp\left(\zeta(|x|)|x|^2\right)
=\left\{\frac{|x|}{2}\,\zeta^\prime(|x|)+\zeta(|x|)\right\}
\left(ax^\perp\right)
\label{auxi}
\end{equation}
where
$\zeta\in C^\infty([0,\infty); [0,1])$ satisfies
$\zeta(r)=1\,(r\leq R)$ and $\zeta(r)=0\,(r\geq 2R)$,
where $R\geq 1$ is fixed such that
$\mathbb R^2\setminus\Omega\subset B_R(0)$.
Then we have
\[
w|_{\partial\Omega}=ax^\perp, \quad
\mbox{div $w$}=0, \quad
x^\perp\cdot\nabla w-w^\perp
=\mbox{div $(w\otimes x^\perp-x^\perp\otimes w)$}=0.
\]
We will find a solution of the form
$u=\widetilde u+w$, where $\widetilde u$ should obey
\begin{equation}
\left\{
\begin{array}{l}
-\Delta\widetilde u
-a\left(x^\perp\cdot\nabla\widetilde u-\widetilde u^\perp\right)+\nabla p
=f+\Delta w, \qquad
\mbox{div $\widetilde u$}=0 \quad\mbox{in $\Omega$},  \medskip \\
\widetilde u|_{\partial\Omega}=0,\qquad
\widetilde u\to 0\quad\mbox{as $|x|\to\infty$}.
\end{array}
\right.
\label{modify}
\end{equation}

As Finn and Smith performed in their paper \cite{FS1}
on the Oseen system (see also Galdi \cite[Section VII.5]{Ga-b}),
for $\varepsilon \in (0,1)$, let us consider the approximate problem
\begin{equation}
\left\{
\begin{array}{l}
\varepsilon u_\varepsilon-\Delta u_\varepsilon
-a\left(x^\perp\cdot\nabla u_\varepsilon-u_\varepsilon^\perp\right)+
\nabla p_\varepsilon
=f+\Delta w, \qquad
\mbox{div $u_\varepsilon$}=0 \quad\mbox{in $\Omega$}, \medskip \\
u_\varepsilon|_{\partial\Omega}=0, \qquad
u_\varepsilon\to 0 \quad\mbox{as $|x|\to\infty$}.
\end{array}
\right.
\label{ep-ext}
\end{equation}
By $C_{0,\sigma}^\infty(\Omega)$ we denote the class of all 
solenoidal vector fields being in $C_0^\infty(\Omega)$.
Let $H_{0,\sigma}^1(\Omega)$
be the completion of $C_{0,\sigma}^\infty(\Omega)$ in $H^1(\Omega)$.
In a usual way
(see, for instance, the proof of Lemma 5.3 of \cite{H06},
in which the problem in each bounded domain $\Omega_\rho$
is first solved by means of the Lax-Milgram theorem
and then the limit $\rho\to\infty$ is considered by using a priori
estimate uniformly in $\rho$),
one can find 
$u_\varepsilon\in H_{0,\sigma}^1(\Omega)$
which satisfies
\begin{equation}
\varepsilon\,\|u_\varepsilon\|_{L^2(\Omega)}^2+
\frac{1}{2}\|\nabla u_\varepsilon\|_{L^2(\Omega)}^2
\leq\frac{1}{2}\|F+\nabla w\|_{L^2(\Omega)}^2
\label{appro}
\end{equation}
and
\[
\varepsilon\langle u_\varepsilon, \varphi\rangle
+\langle\nabla u_\varepsilon, \nabla\varphi\rangle
-a\,\langle x^\perp\cdot\nabla u_\varepsilon-u_\varepsilon^\perp, \varphi\rangle
=\langle f+\Delta w, \varphi\rangle
\]
for all $\varphi\in C_{0,\sigma}^\infty(\Omega)$.
We choose $p_\varepsilon\in L^2_{loc}(\overline{\Omega})$ such that
$\int_{\Omega_{3R}}p_\varepsilon\,dx=0$ and that the pair
$\{u_\varepsilon, p_\varepsilon\}$ satisfies 
$\mbox{\eqref{ep-ext}}_1$ in the weak sense.
Since $f+\Delta w\in C^\infty(\Omega)$,
the interior regularity theory of the usual Stokes system implies that
$u_\varepsilon,\, p_\varepsilon\in C^\infty(\Omega)$.

As in the proof of Theorem \ref{decay}, we take
the same cut-off function $\psi$
together with the Bogovskii operator $B$ in the annulus
$A=\{x\in\mathbb R^2; R<|x|<3R\}$ and set
\[
v_\varepsilon=(1-\psi) u_\varepsilon+B[u_\varepsilon\cdot\nabla\psi], \qquad
q_\varepsilon=(1-\psi) p_\varepsilon.
\]
Then the pair $\{v_\varepsilon,q_\varepsilon\}$ obeys
\begin{equation*}
\begin{split}
\varepsilon v_\varepsilon-\Delta v_\varepsilon
-a\left(x^\perp\cdot\nabla v_\varepsilon-v_\varepsilon^\perp\right)+
\nabla q_\varepsilon
&=g_\varepsilon+(1-\psi)f  \qquad\mbox{in $\mathbb R^2$}, \\
\mbox{div $v_\varepsilon$}&=0  \qquad\mbox{in $\mathbb R^2$},
\end{split}
\end{equation*}
where
\begin{equation*}
\begin{split}
g_\varepsilon
&=\varepsilon B[u_\varepsilon\cdot\nabla\psi]
+2\nabla\psi\cdot\nabla u_\varepsilon
+(\Delta\psi+ax^\perp\cdot\nabla\psi)u_\varepsilon
-\Delta B[u_\varepsilon\cdot\nabla\psi] \\
&\qquad -ax^\perp\cdot\nabla B[u_\varepsilon\cdot\nabla\psi] 
+a B[u_\varepsilon\cdot\nabla\psi]^\perp-(\nabla\psi)p_\varepsilon.
\end{split}
\end{equation*}
Here, note that $(1-\psi)\Delta w=0$ since
$\psi =1\; (|x|\leq 2R)$ and since
$\Delta w=0\; (|x|\geq 2R)$.
We use the fundamental solution \eqref{ep-funda} for the system 
\eqref{rot-st-ep}.
Then, by Proposition \ref{ep-decay} with $\rho=6R$ and
Lemma \ref{f-supp}, we find
\[
v_\varepsilon(x)
=\int_{\mathbb R^2}\Gamma_a^{(\varepsilon)}(x,y)
\left\{g_\varepsilon+(1-\psi)f\right\}(y)\,dy
\]
subject to 
\begin{equation}
\begin{split}
\sup_{|x|\geq 6R}|x||v_\varepsilon(x)|
&\leq C(1+|a|^{-1})
\Big[
\int_{\mathbb R^2}(1+|x|) 
\left|\left\{g_\varepsilon+(1-\psi)f\right\}(x)\right|\,dx  \\
&\qquad\qquad +\sup_{|x|\geq 3R}|x|^3(\log |x|)|f(x)|
\Big]
\end{split}
\label{v-est}
\end{equation}
with $C>0$ independent of $\varepsilon$.
Here, the point is that
a constant vector can be excluded since $u_\varepsilon\in L^2(\Omega)$.

By $\int_{\Omega_{3R}}p_\varepsilon\,dx=0$ and 
$\mbox{\eqref{ep-ext}}_1$ we have
\[
\|p_\varepsilon\|_{L^2(\Omega_{3R})}
\leq C_R\|\nabla p_\varepsilon\|_{H^{-1}(\Omega_{3R})}
\leq 
C_R\left(\|u_\varepsilon\|_{H^1(\Omega_{3R})}+\|F+\nabla w\|_{L^2(\Omega_{3R})}\right),
\]
where $H^{-1}(\Omega_{3R}):=H_0^1(\Omega_{3R})^*$.
This together with \eqref{appro} and the estimate
of the Bogovskii operator (\cite{B}, \cite{BS} and \cite{Ga-b}) lead us to
\begin{equation*}
\begin{split}
\int_A|g_\varepsilon(y)|\,dy
&\leq 2\sqrt 2R\|g_\varepsilon\|_{L^2(A)} \\
&\leq C_R
\left(\|u_\varepsilon\|_{H^1(\Omega_{3R})}+\|p_\varepsilon\|_{L^2(\Omega_{3R})}\right) \\
&\leq C_R\left(\|\nabla u_\varepsilon\|_{L^2(\Omega_{3R})}
+\|F\|_{L^2(\Omega_{3R})}+|a|\right) \\
&\leq C_R(\|F\|_{L^2(\Omega)}+|a|),
\end{split}
\end{equation*}
which combined with \eqref{v-est} implies that
\begin{equation}
|u_\varepsilon(x)|=|v_\varepsilon(x)|
\leq C(1+|a|^{-1})(|a|+\|F\|_{L^2(\Omega)}+[\,f\,])\,|x|^{-1} \quad
(|x|\geq 6R),
\label{far}
\end{equation}
where
\[
[\,f\,]
:=\int_{|x|\geq 2R}|x||f(x)|\,dx
+\sup_{|x|\geq 3R}|x|^3(\log |x|)|f(x)|
\]
and $C=C(R)>0$ is independent of $\varepsilon\in (0,1)$.
By \eqref{appro} we have
\[
\|u_\varepsilon\|_{L^{2,\infty}(\Omega_{6R})}
\leq C\| u_\varepsilon\|_{L^2(\Omega_{6R})}
\leq C_R\|\nabla u_\varepsilon\|_{L^2(\Omega_{6R})}
\leq C_R(\|F\|_{L^2(\Omega)}+|a|),
\]
which together with \eqref{far} yields
\[
u_\varepsilon\in L^{2,\infty}(\Omega), \qquad
\|u_\varepsilon\|_{L^{2,\infty}(\Omega)}
\leq C\left\{1+|a|+(1+|a|^{-1})\left(\|F\|_{L^2(\Omega)}+[\,f\,]\right)\right\}
\]
with $C=C(R)>0$ independent of $\varepsilon\in (0,1)$.
Hence, there is 
$\widetilde u\in L^{2,\infty}(\Omega)\cap H^1_{loc}(\overline{\Omega})$ with
$\nabla\widetilde u\in L^2(\Omega)$
such that,
as $\varepsilon\to 0$ along a subsequence,
\begin{equation*}
\begin{split}
&u_\varepsilon\to \widetilde u \quad
\mbox{weakly-star in $L^{2,\infty}(\Omega)$}, \qquad
\nabla u_\varepsilon \to \nabla\widetilde u \quad
\mbox{weakly in $L^2(\Omega)$}, \\
&u_\varepsilon\to \widetilde u \quad
\mbox{weakly in $H^1(\Omega_\rho)$}, \qquad
u_\varepsilon\to \widetilde u \quad
\mbox{strongly in $L^2(\Omega_\rho)$},
\end{split}
\end{equation*}
for every $\rho\geq R$ and, thereby,
\[
\langle\nabla\widetilde u, \nabla\varphi\rangle
-a\,\langle x^\perp\cdot\nabla\widetilde u-\widetilde u^\perp, \varphi\rangle
=\langle f+\Delta w, \varphi\rangle
\]
holds for all $\varphi\in C_{0,\sigma}^\infty(\Omega)$,
as well as $\mbox{div $\widetilde u$}=0$.
We fix $\rho$ and use the trace inequality
\[
\|u_\varepsilon-\widetilde u\|_{L^2(\partial\Omega_\rho)}
\leq C\|u_\varepsilon-\widetilde u\|_{L^2(\Omega_\rho)}^{1/2}
\|u_\varepsilon-\widetilde u\|_{H^1(\Omega_\rho)}^{1/2}
\]
to see that $\widetilde u|_{\partial\Omega}=0$.
Since 
$\Delta\widetilde u+a\,\big(x^\perp\cdot\nabla\widetilde u-\widetilde u^\perp\big)
+f+\Delta w\in H^{-1}(\Omega_\rho)$
for every $\rho\geq R$,
we find an associated pressure $p\in L^2_{loc}(\overline{\Omega})$
such that the pair $\{\widetilde u,p\}$ solves 
$\mbox{\eqref{modify}}_1$ in the weak sense.
The interior regularity theory of the Stokes system
concludes that $\{\widetilde u,p\}$ is smooth and, therefore,
so is $u:=\widetilde u+w$.
Both estimates in \eqref{weak-est} are obvious.

Let us apply Theorem \ref{decay} to $\{u,p\}$ with $\nabla u\in L^2(\Omega)$
as well as $u\in H^1_{loc}(\overline{\Omega})$,
$p\in L^2_{loc}(\overline{\Omega})$.
Since $u\in L^{2,\infty}(\Omega)$, we have all the properties in this theorem
with $u_\infty=0$.
We denote $p-p_\infty$ by the same symbol $p$ so that
$\{u,p\}$ is the desired solution.
By $u|_{\partial\Omega}=ax^\perp$ we have
\[
\beta=\int_{\partial\Omega}\nu\cdot u\,d\sigma=0, \quad
\int_{\partial\Omega}
y^\perp\cdot\left\{(u\otimes y^\perp)\nu\right\}\,d\sigma_y
=a\!\int_{\partial\Omega}(\nu\cdot y^\perp)|y|^2\,d\sigma_y=0,
\]
and thereby \eqref{asym-rep} implies \eqref{rot-u-expan}.
Finally, asymptotic representation of the pressure
is given by \eqref{p-repre}, in which
$\{u_\infty, p_\infty\}=\{0,0\}$.
Since $\beta=0$ and $u|_{\partial\Omega}=ay^\perp$, we conclude \eqref{rot-p-expan}.
The proof is complete.
\hfill
$\Box$
\bigskip

\noindent
{\em Acknowledgments.}
I would like to thank Professor M. Yamazaki for stimulating discussions.
Thanks are also due to Professor G.P. Galdi for informing me of
\cite{HW}, and to Dr. J. Guillod for sending \cite{GW} to me.
This work is partially supported by
Grant-in-Aid for Scientific Research, No. 24540169, from JSPS.

\end{document}